\documentclass[12pt, oneside,emlines]{amsart}
\usepackage{amssymb,latexsym,xy,eucal,mathrsfs}
\usepackage{MnSymbol}
\textwidth=15.2cm \textheight=24cm \theoremstyle{plain}
\newtheorem{lemma}{Lemma}[section]
\newtheorem{theorem}[lemma]{Theorem}
\newtheorem{proposition}[lemma]{Proposition}
\newtheorem{corollary}[lemma]{Corollary}

\theoremstyle{definition}
\newtheorem{example}[lemma]{Example}
\newtheorem{remark}[lemma]{Remark}

\newtheorem{definition}[lemma]{Definition}

\usepackage[bottom=0.1in,top=1.5in, left=1.5in, right=0.3in]{geometry}
\numberwithin{equation}{section}  \voffset
-55truept \hoffset -55truept
\begin{document}

\baselineskip 15truept
\title{On Generalized p.q.-Baer $*$-rings}

\date{}

\author{ Sanjay More and  Anil Khairnar }
\address {\rm Department of Mathematics, Prof. Ramkrishna More College, Akurdi  Pune-411044, India.}
\email{\emph{sanjaymore71@gmail.com}} 

\address{\rm Department of Mathematics, Abasaheb Garware College, Pune-411004, India.}
\email{\emph{anil\_maths2004@yahoo.com}}


\maketitle {\bf \small Abstract:} {\small 
  We introduced the class of weakly generalized p.q.-Baer $*$-rings. It is proved that under some assumptions every weakly generalized p.q.-Baer $*$-ring  can be embedded in generalized p.q.-Baer $*$-ring.  We proved that a generalized p.q.-Baer $*$-rings has partial comparability. If a generalized p.q.-Baer $*$-ring satisfies the parallelogram law then it is proved that every pair of projections has an orthogonal decomposition. A separation theorem for generalized p.q.-Baer $*$-rings is obtained. As an application of spectral theory, it is proved that generalized p.q.-Baer $*$-rings have  a sheaf representation with injective sections. 
   
}


\section{Introduction} 
	  
Let $R$ be a ring and $S$ be a nonempty subset of $R$. We write
$r_R(S)= \{a \in R ~|~ s a = 0, ~\forall ~s \in S \}$, and is called the
{\it right annihilator} of $S$ in $R$, and $l_R(S)= \{a \in R ~|~  a s = 0, ~\forall ~s \in S \}$, is the { \it left annihilator} of $S$ in $R$.         
A {\it $*$-ring $R$} is a ring equipped with an involution $x \rightarrow x^* $,
that is, an additive anti-automorphism of the period at most two.
An element $e$ of a $*$-ring  $R$ is called a {\it projection} if it is self-adjoint 
(i.e. $e=e^*$) and idempotent (i.e. $e^2=e$). 
A $*$-ring $R$ is said to be a {\it  Rickart $*$-ring},
if for each $x \in R$, $r_R(\{x\})=eR$, where $e$ is a projection
in $R$. A $*$-ring $R$ is said to be a {\it Baer $*$-ring}, if for every nonempty subset $S$ of $R$, $r_R(S)=eR$, where $e$ is a projection in $R$.

In \cite{Kap}, Kaplansky introduced Baer rings and Baer $*$-rings to
abstract various properties of $AW^*$-algebras, von Neumann
algebras, and complete $*$-regular rings. 
Early motivation for studying $*$-rings came from the rings of operators.
If $\mathscr{B}(H)$ is the set of all bounded linear operators on a (real or complex) Hilbert space $H$,
then each $\phi \in \mathscr{B}(H)$ has an adjoint, $adj (\phi) \in \mathscr{B}(H)$,
and $\phi \rightarrow adj(\phi)$ is an involution on the ring $\mathscr{B}(H)$. One can refer \cite{Anil2,Anil3,Anil4} for recent work on rings with involution.

In \cite{Bir}, Birkenmeier et al. introduced quasi-Baer $*$-rings as a generalization of Baer $*$-rings.
A $*$-ring $R$ is said to be a {\it quasi-Baer $*$-ring}, if
for every ideal $I$ of $R$, $r_R(I)=eR$, where $e$ is a projection in $R$.           
Birkenmeier et al. \cite{Bir5} introduced principally quasi-Baer (p.q.-Baer) $*$-rings as a generalization of quasi-Baer $*$-rings.
A $*$-ring $R$ is said to be a {\it p.q.-Baer $*$-ring}, if for every principal right ideal $aR$ of $R$, $r_R(aR)=eR$, where $e$ is a projection in
$R$. It follows that $l_R(aR)=Rf$ for
a suitable projection $f$. A $*$-ring $R$ is called a {\it generalized (principally) quasi-Baer $*$-ring } if, for every (principal) ideal $I$ of $R$, the right annihilator of $I^n$ is generated as a right ideal by a projection for some positive integer $n$ (depending on $I$). That is, there exists a projection $e\in R$ such that $r_R(I^n)=eR$,  \cite{ahmadiquasi2020}.
The class of generalized p.q.-Baer $*$-rings is broader than the classes of Baer $*$-rings and Rickart $*$-rings, see [\cite{ahmadiquasi2020}, \cite{Bir5}, \cite{Anil4}]. 

It is known that the set of projections of a Rickart $*$-ring forms a lattice \cite{Ber}. Also, there exists a Rickart $C^*$-algebra whose projection lattice is not complete \cite{Ber}.
 In the second section of this paper, we prove that the set of projections of a p.q.-Baer $*$-ring $R$ forms a lattice with the partial order $e \leq f$ if and only if $e=ef$, for projections $e,f \in R$.  Also, we give a characterization of quasi Baer $*$-ring in terms of completeness of the lattice of projections of a p.q.-Baer $*$-ring. 
 
 Let $R$ be a $*$-ring, $x \in R$. We say that $x$ possesses a {\it central cover}  if there exists a smallest central projection $h$ such that $hx=x$; that is, (i) $h$ is a central projection, (ii) $hx=x$, and (iii) if $k$ is a central projection with $kx=x$ then $h \leq k$. If such a projection $h$ exists, it is unique; we call it the central cover of $x$, denoted by $h=C(x)$. In \cite{Ber}, it is proved that in a Baer $*$-ring every element possesses a central cover. In \cite{Anil}, authors extended this result for p.q.-Baer $*$-rings and proved that every element of a p.q.-Baer $*$-ring possesses a central cover. In the second section of this paper, we introduce the concept of the generalized central cover of an element of a $*$-ring. We prove that every element of a generalized p.q.-Baer $*$-ring possesses a generalized central cover.  We give a characterization of generalized p.q.-Baer $*$-ring in terms of the existence of generalized central covers of all elements.

 In \cite{Ber}, weakly Rickart $*$-rings are introduced and it is proved that under some assumptions every weakly Rickart $*$-ring can be embedded in a Rickart $*$-ring with preservation of right projections. In \cite{Anil, Anil4}, authors extended this theory of Rickart $*$-rings, and Baer $*$-rings to p.q.-Baer $*$rings. 
 In section three of this paper, we introduce the concept of weakly generalized p.q.-Baer $*$-ring and give a characterization of generalized p.q.-Baer $*$-ring in terms of weakly generalized p.q.-Baer $*$-ring. It is proved that under some assumptions every weakly generalized p.q.-Baer$*$-ring can be embedded in a generalized p.q.-Baer $*$-ring with preservation of generalized central covers. 

 Let $A$ be a $*$-ring. The elements $e, f \in P(A)$ are said to be {\it equivalent} (written as $e \sim f$), 
 	if there exists $w\in A$
 	such that $w^*w=e$ and $ww^*=f$. A $*$-ring whose projections form a lattice is said to satisfy the {\it parallelogram law} if $e-e\wedge f \sim  e\vee f-f$ for every pair of projection $e$ and $f$. Note that Baer $*$-ring may fail to satisfy the parallelogram law but every von Neumann algebra satisfies the parallelogram law.  Let $e, f$ be projections in $R$, we say that $e$ is {\it dominated } by $f$,
 	written as $e \lesssim f$, if $e \sim g \leq f$ for some $g \in P(R)$ (the set of all projections in $R$).  
 Projections $e$ and $f$ in a $*$-ring $R$ are said to be {\it generalized comparable} if there exists a central projection $h$ such that $he \lesssim hf$ and $(1-h)f \lesssim (1-h)e$.
 	We say that a $*$-ring $R$ has {\it generalized comparability} $(GC)$ if every pair of projections is generalized comparable.  
  Projections $e$ and $f$ in a $*$-ring $R$ are said to be {\it very orthogonal} if there exists central projection $h$ such that $he = e$ and $hf = 0$. Projections $e$ and $f$ in a $*$-ring $R$ are said to be {\it partially comparable} if there exist non-zero projections $e_{0}$ and $f_{0}$ such that $e_{0} \le  e$,  $f_{0} \le f$, and $e_{0} \sim  f_{0}$.
 	We say that a $*$-ring $R$ has  {\it partial comparability} $(PC)$ if $eRf \neq  0$ implies  $e$ and  $f$ are partially comparable. 
  A $*$-ring $R$ has { \it orthogonal GC} if every pair of orthogonal projections are generalized comparable.	
   
 In section four we prove that in a generalized p.q.Baer $*$-ring, every pair of projections, parallelogram law holds provided one of the projections is central. 
 Also, we give a characterization of a generalized p.q.-Baer $*$-ring to satisfies the parallelogram law. It is proved that generalized p.q.-Baer $*$-ring has partial comparability (PC). Also, we obtained an orthogonal decomposition of two projections in a generalized p.q.-Baer $*$-ring satisfying parallelogram law.

 In \cite{Anil4}, authors introduced the concepts of 
 central strict ideals and prime central strict ideals
 in a p.q.-Baer $*$-ring.
  A separation theorem for p.q.-Baer $*$-rings is obtained. 
 In the section five of this paper, we introduce the concepts of generalized 
central strict ideals and prime generalized central strict ideals
in a generalized p.q.-Baer $*$-ring.
We prove that any two prime generalized central strict ideals of a generalized p.q.-Baer $*$-ring are incomparable. We give a characterization of prime generalized central strict ideals in a generalized p.q.-Baer $*$-ring.
We prove the separation theorem for generalized p.q.-Baer $*$-rings.

  A sheaf representation of a ring $R$ is just the sheaf space over a topological space $X$ such that $R$ can be recovered as the ring of continuous global sections.  
  Sheaf representations of rings have proved
  valuable either as tools for algebraic geometry or simply as structure theorems generalizing the representation of certain commutative rings as rings of continuous functions.
  When a ring $R$ is commutative the most widely used
  representations are the Grothendieck sheaf \cite{Gro1} and the Pierce sheaf \cite{Pierce}. It is quite natural to try extending these results to the general case where $R$
  is any ring. Numerous attempts have been  made to construct
  a sheaf representation for a general ring.
   An application of spectral theory to the sheaf representation of various algebraic 
  structures is studied by many authors, such as  
  \cite {Dau,  Shin, Tha}. 
  By a ``sheaf representation'' 
  of a ring, they mean a sheaf representation whose base space is $Spec(R)$ (collection of all prime ideals of $R$)
  and whose stalks are the quotients $R/O(P)$, where $P$ is a prime ideal of $R$ and 
  $O(P)=\{a \in R~|~aRs=0$ for some $s \in R\backslash P\}$.
  Almost all of these authors gave a sheaf representation of non-commutative rings over  $Spec(R)$ having stalks $R/O(P)$.  
  In the sixth section of this paper, it is proved that the set of all prime generalized central strict ideals of a generalized p.q.-Baer $*$-ring,  $\Sigma(R)$, carries the hull-kernel topology.  We obtained a sheaf representation of generalized p.q.-Baer $*$-rings.
 
 For recent work on $*$-rings and Rickart rings one can see \cite{Bay, Kara, Saad, Sidd}.
 
	\section{GENERALIZED P.Q.-BAER *-RING}

	 Let $P$ be a poset and $a,b \in P$. The {\it join} of $a$ and $b$, denoted by
		$a\vee b$ is defined as $a \vee b = \sup~ \{a, b\}$. The {\it meet} of $a$ and $b$, denoted by $a\wedge b$ is defined as $a \wedge b= inf~ \{a, b\}$. Recall the following result from \cite{Ber}, which says projections of Rickart $*$-rings form a lattice.
	          
	\begin{proposition} [{\cite[Proposition 7, page 14]{Ber}}] \label{tc1s4pr8} The projections of a Rickart $*$-ring form a lattice. Explicitly, $e \vee f=f+RP(e(1-f))$, $e \wedge f=e- LP(e(1-f))$ for every pair of projections. 
	\end{proposition}
	
		 In the following result, we prove that the set of projections in a p.q.-Baer $*$-rings forms a lattice. 
	\begin{proposition} \label{s2pr2}
		The set of projections of a p.q.-Baer $*$-ring form a lattice. Explicitly, for projections $e$ and $f$ in a p.q.-Baer $*$-ring; $e \vee f=f+C(e(1-f))$ and $e \wedge f=e-C(e(1-f))$. 
	\end{proposition}
	\begin{proof} Let $R$ be p.q.-Baer $*$ring and $x=e(1-f)=(e-ef)$. Then $xf=e(1-f)f=(ef-ef) =0$.
		Since $R$ is a p.q.-Baer $*$-ring, we have $r(xR)= (1-C(x))R$.
		Let $g=1-C(x)$, therefore $C(x)=1-g$.
		We prove that $e \vee f=f+(1-g)=f+C(e(1-f))$.
		Since $xf=0$, we have $x(1-f)=x-xf=x$.
		Hence $C(x)\leq1-f$. This implies $1-g\leq1-f$, which is equivalent to $f \leq g$.
		This gives $f(1-g) =0$, and hence $h=f+(1-g)$ is a projection.
		Since $hf=f+(1-g)f =f$, we have $ f \leq h $.
		Now we prove $e \leq h$.
		Since $xC(x)=x$ gives $x(1-g)=x$.
		We have $x-xg=x$, that is $xg=0$. Therefore $(e-ef)g=0$.
		Hence $eg=efg$. This implies $eh=e(f+1-g)=ef+e-eg=ef+e-efg$. 
		Since $f(1-g) =0$, we have $e+ef(1-g)=e$.
		Therefore $e \leq h$.
		Hence $h$ is an upper bound of $e$ and $f$.
		Now to prove $h$ is least.
		Let $k$ be a projection such that $e \leq k$ and $f \leq k$.
		As $xk=(e-ef)k=ek-efk=e-ef=x$.
		So $C(x) \leq k $, thus $1-g \leq k$.
		Also $f\leq k $.
		Therefore $f+(1-g) \leq k$, which gives $ h \leq k$.
		Hence $e \vee f= h=f+(1-g)=f+C(x)=f+C(e(1-f))$
		Similarly, $e \wedge f= 1-(1-f) \vee (1-e)=1-\{(1-e)+C[(1-f)(1-(1-e))]\}
		= 1-\{1-e+C[(1-f)e]\}=e-C(e(1-f))$.
	\end{proof}	
	 A poset (lattice) is said to be {\it complete} if every subset has a supremum.	
	 The following result characterizes quasi-Baer$*$-rings in terms of completeness of the lattice of projections of a p.q.-Baer $*$-rings.  	
	\begin{proposition} \label{s2pr3} The following conditions on a $*$-ring $R$ are equivalent
		\begin{enumerate}
			\item $R$ is quasi-Baer $*$-ring ;
			\item $R$ is p.q.-Baer $*$-ring whose projections form a complete lattice.
		\end{enumerate}
	\end{proposition}
	\begin{proof}
	 (1) $\Rightarrow$ (2):
		If $R$ is a quasi Baer $*$-ring then it is a p.q.-Baer $*$-ring.
		So it is enough to prove the set of projections in $R$, denoted by $\tilde{R}$ is a complete lattice. By Proposition \ref{s2pr2}, $\tilde{R}$ is a lattice.
		Let $S \subseteq \tilde{R}$.
		We claim that $\bigcap_{e\in S}r(eR) =(1-supS)R$.
		Suppose $\bigcap_{e \in S} r(eR)=(1-g)R$.
		We will prove that $g= supS$.
		Now $1-g\in r(eR)$ for all $e\in S$. Therefore $e(1-g)=0$ for all $e\in S$, this implies $ e\leq g$ for all $e\in S$.
		Thus $g$  is an upper bound of $S$.
		Let $k\in \tilde{R}$ be such that $e\leq k$ for all $e\in S$.
		Therefore $ e(1-k) =0$ for all $e\in S$. Which gives 
		$ 1-k \in r(eR)$ for all $e\in S$, that is $1-k \in \bigcap_{e\in S}r(eR)$.
		Hence $1-k \in (1-g)R$, this implies $1-k=(1-g)(1-k)$. That is	$ g=gk$. Hence $g\leq k$. Thus $supS =g$.
		Similarly, $\bigcap_{e\in S}r((1-e)R) =(infS)R$.\\
		(2) $\Rightarrow$ (1):
		To prove $R$ is quasi Baer $*$-ring. Take an ideal $I$ in $R$.As $R$ is p.q.-Baer $*$-ring.Therefore for $x_i\in I,r(x_iR)=g_iR$ for some projection $g_i$.By assumption $g=sup\{g_i\}$ exists.We will prove that $r(I)=gR$ . Now $x\in r(I)$ if and only if $x_ix=0$ for all $i$  if and only if $x\in g_iR$ if and only if $x=g_ix$ for all $i$  if and only if $x(1-g_i)=0$ for all $i$ if and only if  $x(1-g)=0$ if and only if $x=gx$ if and only if $x\in gR$. Hence $r(I)=gR$. Thus $R$ is quasi Baer $*$-ring.
	\end{proof} 	

	 An involution of a $*$-ring $R$ is called {\it semi-proper} if $xRx^* =0$ implies $ x=0$ for any $x\in R$, \cite{Birnew}.
	An involution of a $*$-ring $R$ is called {\it quasi-proper} if $xRx^* =0$ implies $ x^n=0$ for some  $n\in \mathbb N$, \cite{ahmadiquasi2020}. 

	\begin{proposition}\label{s3pr2}
		Let $R$ be a generalized p.q.-Baer $*$-ring. Then, 			
			 for every $x\in R$ there exists $n\in \mathbb N$ such that $r(xR)^n\cap (x^*)^n R=\{0\}$;

	\end{proposition}
	\begin{proof}
		 Let $x\in R$.
		Since $R$ is a generalized p.q.-Baer $*$-ring, 
		there exists $n\in \mathbb N$ and a projection $e\in R$  such that $ r(xR)^n=eR$.
		We prove that $r(xR)^n\cap (x^*)^n R=\{0\}$.
		Let $y\in r(xR)^n\cap (x^*)^n R$.
		So $y\in r(x^*)^nR$ implies $y =(x^*)^ns$ for some $s\in R$.
		Also, $y\in r(xR)^n =eR$. Therefore $y=ey =e(x^*)^ns =e(x^n)^*s = (x^ne)^*s=0$.
		Thus $r(xR)^n\cap (x^*)^n R=\{ 0\}$.

	\end{proof}
	
	 	\begin{proposition} \label{s3pr3} Let $R$ be a generalized p.q.-Baer  $*$-ring, and $x \in R$. Then there exists a central projection $e$ such that 
	\begin{enumerate}	
		\item $x^ne=x^n$ for some $n\in \mathbb N$;
		\item For some $n\in \mathbb N $, $(xR)^ny=\{0\}$ if and only if $ey=0$.
	\end{enumerate}
\end{proposition}
\begin{proof}
	(1): Let $x\in R$. Then $r(xR)^n =gR$ for some central projection $g$ and $n\in \mathbb N $. Let $e=1-g$. Since $x^ng =0$, we have $x^ne =x^n(1-g) =x^n-x^ng =x^n$.
	Also, $e$ is central as $g$ is central.\\
	(2): Let $n\in \mathbb N $ be such that $(xR)^ny=\{0\}$. Then $ y\in r(xR)^n =gR$. This implies $y=gy$, that is $y=(1-e)y$. Hence $ey=0$. Conversely, suppose $ey=0$. Then $y=y-ey=(1-e)y=gy$. Since $r(xR)^n =gR$ for some central projection $g$ and $n\in \mathbb N $. Therefore $(xR)^ny= (xR)^ngy=\{0\}$.
\end{proof}
\begin{remark}\label{s3rm2}  We can also prove $y(Rx)^n =\{0\}$ if and only if $yf=0$,  where $fx^n=x^n$.
\end{remark}
\begin{proposition} [{\cite[Proposition 3, page 35]{Ber}}] \label{tc1s4prop8}
	In a Baer $*$-ring $A$, every element $x$ has a central cover. Explicitly, if $e=RP(x)$ and $f=LP(x)$, then $C(x)=C(e)=C(f)$ and $(1-C(x))A=l_A(Ax)=l_A(Ae)=l_A(Af)$, $(1-C(x))A=r_A(xA)=r_A(eA)=r_A(fA)$.
\end{proposition}
In \cite{Anil}, the authors proved results for p.q.-Baer $*$-rings analogous to the above results of Baer $*$-rings. They generalized Proposition \ref{tc1s4prop8} for p.q.-Baer $*$-rings and introduced
the concept of a weakly p.q.-Baer $*$-ring analogous to the concept of a weakly Rickart $*$-ring.

We introduce the concept of generalized central cover in a $*$-ring.
\begin{definition} Let $R$ be a $*$-ring and $x\in R$. Then the smallest central projection $e$ is called the {\it generalized central cover} of $x$ if $x^ne =x^n$ for some $n\in  N $. We write it as $GC(x)=e$.
\end{definition}
\begin{remark}\label{s3rm3} If $x$ is a central element  in the above proposition \ref{s3pr3} then $e$ is the smallest and hence it is $GC(x)$.
\end{remark}
\begin{proof}
	We have $x^ne=x^n$.
	Suppose $x^nk=x^n$ for some central projection $k$.
	Therefore $x^n(1-k) =0$, that is  $(xR)^n(1-k) =\{0\}$ since $x$ is a central element.
	Therefore $1-k\in r(xR)^n) =gR$.
	Hence $1-k =g(1-k)$ which implies $ 1-k \leq g$ that is 	$  1-g \leq k$ that implies $e\leq k$.
	Hence $e$ is the smallest central projection such that  $x^ne=x^n$.
	Thus $GC(x)=e$.
\end{proof}

	We have proved that in a generalized  p.q.-Baer  $*$-ring $R,~GC(x)$ exists for every central element $x\in R$. Further if $x$ is not central then $GC(x)=1$ when $x$ is not nilpotent and $GC(x)=0$ when $x$ is  nilpotent. Thus, in a  generalized  p.q.-Baer  $*$-ring $R,GC(x)$ exists for every  element $x\in R$. 		 			

In a  generalized  p.q.-Baer  $*$-ring $R$ there exists $x\in R$ such that $C(x)\neq GC(x)$ for $x\in R$.
\begin {example} Let $M_{2}(\mathbb{C})$ be a $*$-ring with conjugate transpose as involution. Consider the ring $R=M_{2}(\mathbb{C})\times M_{2}(\mathbb{C})$ with involution  $*$ defined component-wise.  Let us find projections in $R$. Now $E=(E_{1} ,E_{2}) \in R$ is a projection if and only if $E^2=E ,E^*=E$. So $(E_{1} ,E_{2})\dot (E_{1} ,E_{2})=(E_{1} ,E_{2}) ,(E_{1} ,E_{2})^* =(E_{1} ,E_{2})$.Thus $ (E_{1}^2 ,E_{2}^2)=(E_{1} ,E_{2}) , (E_{1}^* ,E_{2}^*)=(E_{1} ,E_{2})$. This gives 
   $E_{1}^2 =E_{1} ,E_{2}^2=E_{2} , E_{1}^*=E_{1} ,E_{2}^*=E_{2}$. Hence both $E_{1} ,E_{2}$ are projections in $M_{2}(\mathbb{C})$. Note that $(E_{1} ,E_{2}) \in R$ is central element if and only if $E_{1} ,E_{2}$ are central elements in $M_{2}(\mathbb{C})$. Let us find central elements in $M_{2}(\mathbb{C})$. As $A\in M_{2}(\mathbb{C})$ is central,  $AX=XA$ for every $X\in M_{2}(\mathbb{C})$.\\ Let $A= \begin {bmatrix} 
   a& b\\c&d
   \end{bmatrix}  $. Thus for  $X= \begin {bmatrix} 
   1& 0\\0&0
\end{bmatrix}  $ we have $ \begin {bmatrix} 
a& b\\c&d
\end{bmatrix}    \begin {bmatrix} 
1& 0\\0&0
\end{bmatrix}   =\begin {bmatrix} 
1& 0\\0&0
\end{bmatrix}   \begin {bmatrix} 
a& b\\c&d
\end{bmatrix}$. \\ Therefore $\begin {bmatrix} 
a& 0\\c&0
\end{bmatrix} =\begin {bmatrix} 
a& b\\0&0
\end{bmatrix}  $ which gives $b=c=0$. So $A= \begin {bmatrix} 
a& 0\\0&d
\end{bmatrix}  $. For $X= \begin {bmatrix} 
1& 1\\0&1
\end{bmatrix} , AX=XA $ gives $\begin {bmatrix} 
a& 0\\0&d
\end{bmatrix}   \begin {bmatrix} 
1& 1\\0&1
\end{bmatrix}=\begin {bmatrix} 
1& 1\\0&1
\end{bmatrix} \begin {bmatrix} 
a& 0\\0&d
\end{bmatrix}  $. Thus $\begin {bmatrix} 
a& a\\0&d
\end{bmatrix} =\begin {bmatrix} 
a& d\\0&d
\end{bmatrix}$ which implies $a=d$. Hence $A=\begin {bmatrix} 
a& 0\\0&a
\end{bmatrix}$ is a central element in $ M_{2}(\mathbb{C})$. So central element $E=\begin {bmatrix} 
a& 0\\0&a
\end{bmatrix}$ is a projection if $E^2=E ,E^*=E$. But $E^2=E$ gives $\begin {bmatrix} 
a& 0\\0&a
\end{bmatrix}\cdot \begin {bmatrix} 
a& 0\\0&a
\end{bmatrix} =\begin {bmatrix} 
a& 0\\0&a
\end{bmatrix}$. Thus $\begin {bmatrix} 
a^2& 0\\0&a^2
\end{bmatrix} =\begin {bmatrix} 
a& 0\\0&a
\end{bmatrix}$ gives $a^2=a$. Further, $E^*=E$ gives $\bar{a}=a$.Thus $a=0$ or $a=1$. Therefore the central projections in $ M_{2}(\mathbb{C})$ are $0 $ or $I$. So all central projections in   $R=M_{2}(\mathbb{C})\times M_{2}(\mathbb{C})$ are $(0,0),(0,I),(I,0),(I,I)$. \\ Take $A=(I,N)\in R$ where $N=\begin {bmatrix} 
0& 1\\0&0
\end{bmatrix} \in M_{2}(\mathbb{C})$. Let us find $C(A)$, the central cover of $A$. We need to find the smallest central projection $E=(E_{1},E_{2})$ such that $AE=A$. Thus $(I,N)  (E_{1},E_{2})=(I,N)$. So $(E_{1},NE_{2})=(I,N)$ gives $E_{1}=I , NE_{2}=N$. But $E_{2}=0$ does not satisfy $NE_{2}=N$. Hence $E_{2}=I$. Therefore only central projections in $R$ such that $AE=A$ is $E=(I,I)$ and hence $C(A)=(I,I)$. Now find $GC(A)$. We have to find the smallest central projection $E=(E_{1},E_{2})$ such that $A^nE=A^n$ for some $n\in \mathbb{N}$.  Note that $N=\begin {bmatrix} 
0& 1\\0&0\end{bmatrix} $. Thus $N^2=0$ and $N^n=0$ for every $n\geq 2$. So for $n\geq 2 , A^n=(I,N)^n=(I^n,N^n)=(I,0)$. Thus $A^nE=A^n$ gives $(I,0) (E_{1},E_{2})=(I,0)$ which implies $(E_{1},0)=(I,0)$ hence   $E_{1}=I$. Therefore $E=(I,0)$ or $(I,I)$ both satisfy $A^nE=A^n$ for $n\geq 2$. Now $(I,0)(I,I)=(I,0)$. Thus $(I,0) \leq (I,I)$. Hence $E=(I,0)$ is the smallest central projection such that  $A^nE=A^n$. Thus $GC(A)=(I,0)$. So $C(A)=(I,I)$ and $GC(A)=(I,0)$. Therefore $C(A) \neq GC(A)$ for $A=(I,N)\in R=M_{2}(\mathbb{C})\times M_{2}(\mathbb{C})$ where $N=\begin {bmatrix} 
0& 1\\0&0
\end{bmatrix} \in M_{2}(\mathbb{C})$.
\end{example}
\begin{remark} \label{s3rm4}  In a $*$-ring $R$, $GC(e)=e$, where $e$ is a central projection.
\end{remark}
In the following result, we discuss the interplay between the order of projections and their generalized center cover.
\begin{proposition}\label{s3pr4}
	Let $R$ be a $*$ -ring in which every projection $e$ has a generalized central cover $GC(e)$. Then 
\begin{enumerate}
		
\item 	$e\leq f$ implies $GC(e) \leq GC(f)$;
\item $GC(he) =hGC(e)$ for every central projection $h$.
\end{enumerate}	
\end{proposition}

\begin{proof}
(1): Let $e\leq f$.
	Therefore $e=ef=fe$.
	Let $GC(e) =h$ and $GC(f)=k$.
	Hence $eh=e$ and $fk=f$.
	Thus $e=ef=efk=ek$.
	Therefore $e=ek$ and  $GC(e) =h$ implies $ h\leq k$ that is $GC(e) \leq GC(f)$.\\
(2): Let $GC(he) =k$.
Therefore $(he)^nk=	(he)^n$ for some $n\in\mathbb N$. Thus $hek=he$. We have to prove that $ hGC(e)=k$. Now $hee=he$ implies  $ he \leq e $. Hence $ GC(he) \leq GC(e)$. That is $k \leq GC(e)$ which implies $ k = k.GC(e) $.
Also $he \cdot h=he$ implies $he\leq h$. Therefore $GC(he) \leq GC(h)$ implies $ k\leq h$.
Therefore $k=kh$. Thus $k=khGC(e)$.
As $hek=he$. So $(h-kh)e=0 \Rightarrow (h-kh)GC(e)e=0$. Hence $GC(e)-(h-kh)GC(e)$ is a central projection such that $[GC(e)-(h-kh)GC(e)]e=e$. Hence $GC(e) \leq [GC(e)-(h-kh)GC(e)]$. Therefore $(h-kh)GC(e) GC(e)=0$ and  $e\leq f$ implies $(f-e) e=0$. Thus $(h-kh)GC(e)=0$ implies $ hGC(e) =khGC(e)=k$.
Hence $GC(he) =hGC(e)$. 
\end{proof}
The following proposition is analogous to the result for Baer $*$-ring from \cite{Ber}.  
	\begin{proposition}\label{s3pr5} Let $R$ be a commutative generalized p.q.-Baer $*$-ring then;
	\begin{enumerate}	
		\item	 	$GC(x) =GC(x^*)$ for every $x\in R$;
		\item	 $(xR)^ny =0$ implies $ GC(x)GC(y) =0$ for some $n\in \mathbb N$.
	\end{enumerate}			
\end{proposition}

\begin{proof}
	(1): Let $GC(x^*) =e$ and $GC(x) = f$.
	Hence $(x^*)^me =(x^*)^m$ and  $x^nf=x^n$ for $m ,n \in \mathbb N$.
	Suppose $m < n$.
	Therefore $ex^m=x^m$ that is $x^me=x^m$ this implies $ x^ne=x^n$	that is $	GC(x) \leq e $ hence $ f\leq e$.
	Also 	$f(x^*)^n =(x^*)^n$ and $e(x^*)^n =(x^*)^n$.
	Thus $(e-f)(x^*)^n =0$ implies $ (e-f)(Rx^*)^n =\{0\}$ that is $(e-f)e =0$. Hence $ e=ef$ implies $ e\leq f$.
	Hence $e=f$ i.e. $GC(x) =GC(x^*)$.\\
	(2): Suppose $GC(x) =e$ and 	$GC(y) =f$.
	Therefore $x^ne=x^n$, and $y^mf =y^m$.
	Thus $(xR)^ny =\{0\}$ implies $ ey=0 $ that is $ye=0$. Therefore $(Ry)^me =\{0\}$, which implies $ fe=0$ that is
	$ef =\{0\} $. Hence $ GC(x)\cdot GC(y) =0$.	
\end{proof}

In the following proposition, we give a characterization of the generalized  p.q.-Baer  $*$-ring.
\begin{proposition}\label{s3pr6}
	The following statements are equivalent for a $*$-ring $R$ 
	\begin{enumerate}	
		\item $R$ is  generalized  p.q.-Baer  $*$-ring.
		\item $R$ has a unity element and for every $x\in R$ there exists a central projection $e$ and $n\in \mathbb N$ such that $r(xR)^n =r(eR)^n$.
	\end{enumerate}				
\end{proposition}

\begin{proof}
	$(1)\Rightarrow (2):$ Suppose $R$ is a generalized  p.q.-Baer  $*$-ring.
	Therefore $R$ has a unity element.
	Let $x\in R$. Therefore there exists a central projection $e$ such that 	 $(xR)^ny=\{0\}$ if and only if $ey=0$ for some $n\in  \mathbb N $ if and only if $(eR)^ny=\{0\}$.
	Hence 	$r(xR)^n =r(eR)^n$.\\
	$(2)\Rightarrow (1):$ 
		Let $x\in R$.
	Therefore $r(xR)^n =r(eR)^n$ for some central projection $e$ and $n\in\mathbb N$.
	We claim that $r(eR)^n = (1-e)R$.
	Let $y\in r(eR)^n$. Therefore $(eR)^ny =\{0\}$ implies $ ey=0 $. That is $y=y-ey=(1-e)y$ which implies $ y\in (1-e)R$.
	Hence $r(eR)^n \subseteq (1-e)R$.
	Now let  $y\in (1-e)R$.
	Therefore $y=(1-e)y$ that is $ ey=0 $. Thus $ (eR)^ny =\{0\}$ implies $ y\in r(eR)^n$.
	So $(1-e)R \subseteq r(eR)^n$.
	Thus $r(xR)^n =r(eR)^n =(1-e)R$.
	Hence $R$ is a {\it generalized  p.q.-Baer  $*$-ring}.	
\end{proof}

\section{UNITIFICATION OF WEAKLY GENERALIZED P.Q.-BAER *-RING}
In this section, we introduce the concept of weakly generalized p.q.-Baer $*$-rings. We prove that under some assumptions  weakly generalized p.q.-Baer $*$-ring can be embedded in a  generalized p.q.-Baer $*$-ring.

\begin{definition} A  $*$-ring $R$ is called  {\it weakly generalized   p.q.-Baer  $*$-ring} if for every $x\in R$ there exists central projection  $e$ such that $x^ne=x^n$ and  $(xR)^ny=\{0\}$ if and only if $ey=0$ for some $n\in \mathbb N $.
\end{definition}
In \cite{ahmadiquasi2020}, it is proved that the involution of a generalized p.q.-Baer $*$-ring is quasi-proper. In the following result, we prove that the involution of a weakly generalized p.q.-Baer $*$-ring is quasi-proper.

\begin{proposition}\label{s4pr1}
 If $R$ is a weakly generalized p.q.-Baer  $*$-ring then the involution of $R$ is quasi proper.
\end{proposition}
\begin{proof}
	Let $xRx^* =0$.
	As $x\in R$.
	Therefore there exists central projection $e$ and $n\in\mathbb N$ such that $x^ne=x^n$  and 	$(xR)^ny=\{0\}$ if and only if $ey=0$ for some 	$n\in \mathbb N $.
	We have $(xR)^nx^*=\{0\}$. Therefore 
	$ ex^* =0$ implies $  xe =0$. That is $ (xe)^n =0$ implies $ x^ne =0$. Thus  $ x^n =0$.
	Hence the involution  is quasi proper.	
\end{proof}

The following proposition gives a characterization of generalized  p.q.-Baer  $*$-ring in terms of weakly generalized p.q.-Baer  $*$-ring.

\begin{proposition}\label{s4pr2}
	The following conditions on a $*$-ring $R$ are equivalent
	\begin{enumerate}	
		\item $R$ is  generalized  p.q.-Baer  $*$-ring;
		\item $R$ is weakly generalized p.q.-Baer $*$-ring with unity.
	\end{enumerate}	
\end{proposition}

\begin{proof}
	 $(1)\Rightarrow (2):$ Follows from Proposition \ref{s3pr3}.\\
	$(2)\Rightarrow (1):$ Let $R$ be a weakly generalized p.q.-Baer  $*$-ring with unity. Let $x\in R$.
	Therefore there exists a central projection $e$  such that $x^ne=x^n$  and 	$(xR)^ny=0$ if and only if $ey=0$ for some 	$n\in \mathbb N $.
	Thus $(xR)^ny=0$ if and only if $ey=0$ if and only if $(eR)^ny=0$.
	Hence $r(xR)^n =r(eR)^n$.
	So by  proposition \ref {s3pr6}, $R$ is a generalized  p.q.-Baer  $*$-ring.
\end{proof}
\begin{definition} [{\cite[Definition 3, page 30]{Ber}}] \label{tc1s1d1} Let $R$ be a $*$-ring. We say that $R_1$ is a {\it unitification} of $R$, if there exists an auxiliary ring $K$, called the ring of
	scalars (denoted by $ \lambda, \mu,...$) such that,\\
	(1) $K$ is an integral domain with involution (necessarily
	semi-proper), that is, $K$ is a commutative $*$-ring with unity
	and without divisors of zero (the identity involution is
	permitted),\\
	(2) $R$ is a $*$-algebra over $K$ (that is, $R$ is a left $K$-module such that,	identically 	$1a=a,~\lambda (ab)=(\lambda a)b=a (\lambda b),~ and~(\lambda a)^*=\lambda ^* 	a^*$.\\
	(3) $R$ is a torsion-free $K$-module (that is, $\lambda a=0$ implies $\lambda=0$ or $a=0$).
	Define $R_1=R \oplus K$ (the additive group direct sum), thus
	$(a, \lambda)=(b, \mu)$ means, by the definition that $a=b$ and $\lambda
	=\mu$, and addition in $R_1$, is defined by the formula $(a, \lambda)+(b, \mu)=(a+b, \lambda +
	\mu)$. Define  $(a, \lambda)(b, \mu)=(ab+ \mu a+ \lambda b, \lambda
	\mu)$, $\mu (a, \lambda)=(\mu a, \mu \lambda)$, $(a, \lambda)^*=(a^*, \lambda
	^*)$. Evidently, $R_1$ is also a $*$-algebra over $K$, has unity
	element $(0, 1)$ and is torsion-free; moreover, $R$ is a
	$*$-ideal in $R_1$.
\end{definition}
In the following lemma, we prove that if involution of $R$  is quasi-proper so is of $R_1$. 
\begin{lemma} \label{s4lm1}	If $R$ has a quasi-proper involution then so is the involution in  $R_1$.
\end{lemma}
\begin{proof}
	Suppose $R$ has a quasi-proper involution.
	Therefore $xRx^*=0$ implies $	 x^n=0$ for some $n\in \mathbb N$.
	Let $(a,\lambda ) R_1(a,\lambda )^* =0$.
	Thus  $(a,\lambda )(0,1)(a,\lambda )^* =(0,0)$.
	Hence $(aa^*+\lambda a^*+\lambda^*a ,\lambda\lambda^*) =(0,0)$.
	Therefore $\lambda\lambda^* =0$ implies $\lambda =0$, since $K$ has proper involution.
	Hence $(a,0 )R_1(a^*,0 ) =0$ implies $ (a,0 )(r,0)(a^*,0 ) =(0,0)$ for all $r\in R$. That is $ (ara^*,0 ) =(0,0)$ which implies $ara^*=0$ for all $r\in R $ that is $ aRa^*=0$.
Hence $a^n =0$ for some $n\in \mathbb N$.
	Thus $(a,\lambda )^n =(a,0 )^n=(a^n,0 )=0$.
	Hence $(a,\lambda )R_1(a,\lambda )^* =(0,0)$ implies $(a,\lambda )^n=0$ for some $n\in N$.
	Therefore $ R_1$ has a quasi proper involution.
\end{proof}
In the following result. we will see how to obtain the generalized center cover of an element $R_1$ from the generalized center cover of a corresponding  element of $R$.
\begin{lemma}\label{s4lm2}
	Let $e$ be a projection in $R$. Then $GC(x)=e$ if and only if $GC(x,0)= (e,0)$ in 	$R_1$.	
\end{lemma}
\begin{proof}
	Suppose $GC(x)=e$.
	Therefore $e$ is the smallest central projection  such that $x^ne=x^n$ for some $n\in\mathbb N$.
	Further $(x,\lambda )(e,0)=(xe+\lambda e,0)=(ex+\lambda e,0)=(e,0)(x,\lambda )$.
	Hence $(e,0)$ is in the center of 	$ R_1$.
	Also $(e,0)^* =(e^*,0)=(e,0)$ and $(e,0)^2=(e,0)(e,0)=(e^2,0)=(e,0)$.
	Thus $(e,0)$ is a central  projection in $ R_1$.\\
	Now  $(x,0)^n(e,0)=(x^n,0)(e,0)=(x^ne,0)=(x^n,0)=(x,0)^n$.
	To show $(e,0)$ is the  smallest.
	Let $(x,0)^n(k,0)=(x,0)^n$.
	Therefore $(x^n,0)(k,0)=(x^n,0)$.
	Thus $(x^nk,0)=(x^n,0)$ that is $ x^nk=x^n$ implies $ e\leq k$ since $GC(x)=e$.
	Therefore $e=ek$ that is $(e,0)(k,0)=(ek,0)=(e,0)$.
Hence $(e,0)\leq (k,0)$.
	Therefore $(e,0)$ is the smallest central projection   such that $(x,0)^n(e,0)=(x,0)^n$.
	Hence $GC(x,0)=(e,0)$
	Similarly $GC(x,0)=(e,0)$ gives $GC(x)=e$.
	\end{proof}	
In the following result, we give an existence of central projection which is an upper bound of given two projections of a weakly generalized  p.q.-Baer  $*$-ring. 
\begin{theorem}	\label{s4tm1}	Let $R$ be a weakly generalized  p.q.-Baer  $*$-ring. If $e$ and $f$ are projection in $R$ then there exists central projection $g$ such that $e\leq g$ and $f\leq g$.
\end{theorem}
\begin{proof}
	Let $e$ and $f$ be projections in $R$ .
	Suppose $f$ is central.
	Write $g=f+GC(e-ef)$ and $h=GC(e-ef)$. 
	Thus $g=f+h$.
	Now $h=GC(e-ef)$. As $R$ is weakly generalized  p.q.-Baer  $*$-ring,
	 $((e-ef)R)^nf=\{0\}$ implies $ hf=0$.
	Therefore $g=f+h$ is a central projection.
	We have $gf=f+hf=f$ that is  $f\leq g$.
	Now $h=GC(e-ef)$.
	Therefore $(e-ef)^nh=(e-ef)^n$ for some $n\in\mathbb N$ .
	Hence $e^nh=(e-ef)^n$ since $fh=0$.
	Thus $eh=(e-ef)^n$.
	But $(e-ef)^2 =(e-ef)(e-ef) = e-ef-efe+efef=e-ef$.
	Similarly $(e-ef)^3 =(e-ef)$ and so on $(e-ef)^n=(e-ef)$.
	Thus $eh=(e-ef)$ implies $ eh+ef=e$. That is   $ e=e(h+f)=eg$ implies $ e\leq g$.
	Therefore $g$ is a central projection which is an upper bound of $e$ and $f$.
	Now suppose $e$ and $f$ are any two projections not necessarily central.
	Let $g'$ be a central projection in $R$ (e.g $g'=0$).
	Suppose $e'$ is a  central projection which is an upper bound of $\{e,g' \}$ and $f'$ is  a central projection which is an upper bound of $\{f,g' \}$. Let $g''$ be a central projection which is an upper bound of $\{e',f' \}$.
	Hence $g''$  is an upper bound of $e$ and $f$.
	\end{proof}	
	
In \cite{Ber}, Berberian gives the embedding of weakly Rickart $*$-ring in Rickart $*$-ring. Now, we prove that weakly generalized p.q.-Baer  $*$-ring can be embedded in a generalized  p.q.-Baer  $*$-ring.
\begin{theorem}\label{s4tm2} A weakly generalized p.q.-Baer  $*$-ring $R$ can be embedded in generalized  p.q.-Baer  $*$-ring $R_1$  provided there exists a ring $K$ such that 
	\begin{enumerate}	
		\item	 $K$ is an integral domain with involution;
		\item	 $R$ is $*$-algebra over $K$;
		\item	
		For any nonzero $\lambda\in K$	 there exists a projection $e_\lambda\in R$ such that\\  $\lambda x=0 $ implies $ GC(x)\leq e_\lambda$.
	\end{enumerate}	
\end{theorem}
\begin{proof}
	Let $ R_1 = R\oplus K$ with operations as defined in definition of unitification.
	First, we prove that for any $a\in R$ and  nonzero $\lambda\in K$ there exists a largest central projection $g$ in $R$ such that $(ag+\lambda g)^m=0$ for some $m\in \mathbb N $.
	By condition (3) there exists a projection $e_\lambda\in R$ such that $\lambda x=0$ implies $ GC(x)\leq e_\lambda$.
	Let $GC(a)=e_0$.
	By  theorem \ref{s4tm1} there exists a central projection $e$ such that $e=u\cdot b\{e_0 ,e_\lambda\}$.
	Therefore $ e_0\leq e$ implies $e_0=e_0e=ee_0$.
	Suppose $GC(ae+\lambda e )=h$.
	Then $(ae+\lambda e )^mh=(ae+\lambda e )^m$  and $((ae+\lambda e)R )^my=0$ which implies $ hy=0$.
	Let $g=e-h$.
	Note that $g$ is central as both $e$ and $h$ are central.
	As $(ae+\lambda e )^me=(ae+\lambda e )^m$,
	we have $h\leq e$ that is $ h=he=eh$.
	Also, $eg=e(e-h)=e-eh=e-h=g$ implies $ g\leq e$.
	Hence $(ag+\lambda g)^m=(aeg+\lambda eg)^m=(ae+\lambda e)^mg=(ae+\lambda e)^me-(ae+\lambda e)^mh=(ae+\lambda e)^m-(ae+\lambda e)^m=0$.
	Now we prove that $g$ is the largest.
	Suppose $k$ is a central projection such that $(ak+\lambda k)^m=0$.
	We have $GC(a)=e_0$. 
	Therefore $a^me_0=a^m$ which implies $ a^me_0e=a^me$ that is  $a^me_0=a^me$. Therefore $ a^me=a^m$. 
Hence $a^mek=a^mk $ implies $ ea^mk-a^mk=0$.
	That is $(ek-k)a^mk=0$. 	
Therefore	$ (ek-k)[-\binom{m}{1}a^{m-1}\lambda k-\binom{m}{2}a^{m-2}\lambda^{2}k- \cdots -\lambda^mk]=0$
	 since $(ak+\lambda k)^m=0$.
	 Equating the coefficient of $a^{m-1}$ we get $\lambda(ek-k)m=0$ that is  $\lambda(ek-k)=0$.
	Let $GC(ek-k)=f$.
	Therefore $(ek-k)^nf=(ek-k)^n$ and $((ek-k)R)^ny=0$ implies $fy=0$.
	Thus $\lambda(ek-k)=0$ implies $ \lambda f(ek-k)=0$ i.e $(ek-k)\lambda f=0$.
	Hence $((ek-k)R)^n\lambda f=0$   that is $ f(\lambda f)=0 $ hence $ \lambda f=0 $.
	Therefore by (3) we get that $GC(f) \leq e_\lambda\leq e$.
	Therefore $f\leq e$ since $GC(f)=f$. That is 
	$ f=ef=fe$.
	Now $((ek-k)R)^ne=0$ implies $ fe= 0$ and hence $f=0$.
	Consider $(ek-k)^2=(ek-k)(ek-k)=ekek-ek-kek+k=k-ek$, also
	$(ek-k)^3=(k-ek)(ek-k)=kek-k-ekek+ek=ek-k$.
	In general $(ek-k)=(ek-k)^n=f (ek-k)^n=0$. Therefore $ek=k$, that is $k \leq e$. Now $(ae+\lambda e)^mk=(aek+\lambda ek)^m=(ak+ \lambda k)^m=0$.
	Thus $hk=0$ since $h=GC(ae+ \lambda e)$.
	Hence $kg=k(e-h)=ke-kh=k-0=k$ implies 
	$ k\leq g$ and hence $g$ is largest.
	Let us prove that $R_1$ is a generalized  p.q.-Baer  $*$-ring.
	Let $(a,\lambda)\in R_1$.\\
	Case (i): Let $\lambda  =0$.
	Suppose $GC(a)=e$.
	Therefore $a^ne=a^n$ and $(aR)^ny=\{0\}$ that is $ey=0$.
	Thus $(a,0)^n(e,0)=(a^n,0)(e,0) =(a^ne,0) =(a^n,0)=(a,0)^n$.
	Also $(e,0)$ is a central projection.
	Let $((a,0)R_1)^n(b,\mu) =0$.
	Therefore $(a,0)^n(r,0)^n(b,\mu)=0$ for all $r\in R$.
	Thus $(a^n,0)(r^n,0)(b,\mu) =0$.
	Hence $(a^nr^nb+a^nr^n\mu ,0)=0$ that is 
	 $a^nr^nb+a^ner^n\mu =0$ since $a^ne=a^n$.	
	Therefore we get that $ (aR)^n(b+\mu e)=0$ that is  $ e(b+\mu e)=0$. Hence $(e,0)(b,\mu)=0$.
	Therefore $GC(a,0)=(e,0)$.\\
	Case (ii): Let $\lambda \neq 0$.
	Then there exists a largest central projection $g$ in $R$ such that $(ag+\lambda g)^m=0$  for some 	$m\in \mathbb N $.
	We will prove that $GC(a,\lambda )=(-g,1)$.
	Now $(a,\lambda)^m(-g,1) = (a^m+\binom{m}{1}a^{m-1}\lambda+\cdots+\binom{m}{m-1}a\lambda^{m-1}, \lambda^m)(-g,1)$\\
	= $(-a^mg-\binom{m}{1}a^{m-1}\lambda g-\cdots-\binom{m}{m-1}a\lambda^{m-1}g-\lambda^m g+a^m+	\binom{m}{1}a^{m-1}\lambda+\cdots+\binom{m}{m-1}a\lambda^{m-1} ,\lambda^m)$\\
	= $(-(ag+\lambda g)^m+a^m+\binom{m}{1}a^{m-1}\lambda+\cdots+\binom{m}{m-1}a\lambda^{m-1}, \lambda^m)\\	= (a^m+\binom{m}{1}a^{m-1}\lambda+\cdots+\binom{m}{m-1}a\lambda^{m-1}, \lambda^m)=(a,\lambda )^m$.\\
	Let $((a,\lambda )R_1)^m(b,\mu) =0$. Therefore $((a,\lambda )^m(r,0)^m(b,\mu) =0$ for all $r\in R$ this implies $ (a^m+\binom{m}{1}a^{m-1}\lambda+\cdots+\binom{m}{m-1}a\lambda^{m-1} ,\lambda^m)(r^m,0)(b,\mu) =0$, that is 	
	$	\lambda^m \mu=0 $ hence $  \mu=0 $ since $K$ is an integral domain. Thus $(a^m+\binom{m}{1}a^{m-1}\lambda+\cdots+\binom{m}{m-1}a\lambda^{m-1})r^mb+\lambda^mr^mb=0$ which implies
$ (bR)^m(af+\lambda f)^m =0$, that is $ (af+\lambda f)^m =0$, where $GC(b)=f$.
	But $g$ is the largest projection  in $R$ such that $(ag+\lambda g)^m=0$  some 	$m\in  \mathbb N $.
	Hence $f\leq g$ implies $ f(1-g)=0$.
	Therefore $(-g,1)(b,0)^m=(-g,1)(b^m,0) =(-gb^m+b^m,0) =((1-g)b^m,0)=((1-g)fb^m ,0)=(0,0)=0$.
	Hence $GC(a,\lambda )=(-g,1)$.
	Therefore by Proposition \ref{s4pr2}, $R_1$ is a generalized p.q.-Baer $*$-ring.
\end{proof}	

\section{Parallelogram Law and Generalized Comparability}

In \cite{Ber}, it is proved that in a  Baer  $*$-ring, for every pair of projections $e$ and $f$ parallelogram law holds. In this section, we will prove results on parallelogram law and comparability. 
\begin{proposition}\label{s5pr1} If $R$ is generalized  p.q.-Baer  $*$-ring then for every pair of projections $e$ and $f$ parallelogram law holds provided $e$ or $f$ is central.
\end{proposition}
\begin{proof} We have $e\vee f=f+GC(e(1-f))$ and $e \wedge f =e-GC(e(1-f))$.
	Now $GC(e(1-f))\sim GC(e(1-f))$. Therefore $e-e \wedge f \sim e\vee f-f$. Hence parallelogram law holds in $R$ for  the pair $e$ and $f$.
\end{proof}
 \begin{definition} [{\cite[Definition 2, page 64]{Ber}}]  Let $A$ be a $*$-ring with unity, whose projections form a lattice. Projections $e,f$ in $A$ is said to be at position $p'$ in case $e\wedge (1-f)=(1-e)\wedge f=0$.
  \end{definition}
\begin{proposition}\label{s5pr2} In a generalized  p.q.-Baer  $*$-ring the following are equivalent.
\begin{enumerate}
	\item  The projections $e$ and $f$ are in position $p'$.
	\item $GC(ef)=f$ and $GC(fe)=e$.
\end{enumerate}			  
\end{proposition}
\begin{proof}
$(1) \Rightarrow (2):$ Suppose $e$ and $f$ are in position $p'$. Therefore $e\wedge (1-f) =0 $ and $e\vee(1-f)=1$. Now $GC(ef) =GC(e(1-(1-f)))=e\vee (1-f)-1+f=1-1+f=f$. \\
 $(2) \Rightarrow (1):$ Let $GC(ef)=f$.
 We have  $GC(e(1-(1-f)))=f $ that is $ e\vee (1-f) -1+f =f$, that is 
 $ e\vee(1-f)=1$.
 Also, $GC(fe)=e$.
 Therefore $GC(f(1-(1-e)))=e$ implies  $ f\vee (1-e)-1+e=e$, that is  $f\vee (1-e) =1$ which implies $ e\wedge (1-f)=0$.
 Thus the projections $e$ and $f$ are in position $p'$.
\end{proof}
Certain $*$-rings may be classified through their projection set. This classification entails the following relation in the set of projections.
Recall the following definitions.

\begin{definition} [{\cite[Definition 5, page 4]{Ber}}] Let $A$ be a $*$-ring. The elements $e, f \in P(A)$ are said to be {\it equivalent} (written as $e \sim f$), 
	if there exists $w\in A$
	such that $w^*w=e$ and $ww^*=f$. 
\end{definition}
\begin{definition} [{\cite[Definition 1, page 62]{Ber}}]   A $*$-ring whose projections form a lattice is said to satisfy the {\it parallelogram law} if $e-e\wedge f \sim  e\vee f-f$ for every pair of projection $e$ and $f$.
\end{definition}
In the following result, we give a characterization to satisfy a   parallelogram law by a generalized  p.q.-Baer  $*$-ring.
\begin{proposition}\label{s5pr3} In a generalized  p.q.-Baer  $*$-ring $R$  the following are equivalent.
\begin{enumerate}	
	\item $R$ satisfies the parallelogram law.
	\item If the projections $e$ and $f$ are in position $p'$ then $e \sim f$.
\end{enumerate}
\end{proposition}	
\begin{proof} 
 $(1) \Rightarrow (2):$ Suppose the parallelogram law holds in  $R$. Let the projections $e$ and $f$ be in position $p'$.
Therefore $e\wedge(1-f) =0  $ and $e\vee(1-f)=1$. Now $e-e\wedge f\sim e\vee f-f$. Replace $f$ by $(1-f)$. Thus  $e-e\wedge (1- f)\sim e\vee (1-f)-1+f$ implies $ e-0\sim 1-1+f $. Hence $e\sim f$.\\
 $ (2) \Rightarrow (1):$ 
Let $e$ and $f$ be in position $p'$. Therefore $ e\sim f$. 
 Let the projections $e$ and $f$ be in $R$. Let $GC(ef)=f', GC(fe)=e'$. Now $(ef)^nf=(ef)^n$ implies $GC(ef)\leq f$ that is $f'\leq f $, hence $ f'=f'f=ff'$. Also $(fe)^ne=(fe)^n$ implies  $GC(fe)\leq e$, that is $ e'\leq e$, that is $ e'=e'e=ee'$.
We have $(ef)^nf'=(ef)^n$ and  $(fe)^ne'=(fe)^n$. Now $(e'f')^n=(e'eff')^n=e'(ef)^nf'=e'(ef)^n$. Also, $(e'f')^nf'=(e'eff')^nf'=e'(ef)^nf'=e'(ef)^n$ where $f'$ is the smallest projection with $(ef)^nf'=(ef)^n$. Hence $(e'f')^nf'=(e'f')^n$ implies $ GC(e'f')=f'$. Similarly $GC(f'e')=e'$. Thus by  proposition \ref{s5pr2}   $e'$ and  $f'$ are in position $p'$. Hence by assumption $e' \sim  f'$. We have $e \wedge f =e-GC(e(1-f))$. That is  $e \wedge (1- f) =e-GC(e(1-1+f))$ which implies $ e \wedge (1- f) =e-GC(ef)= e-f'$ that is $ f'=e-e\wedge (1-f)$. Also interchanging $e$ and $f$ in $e \wedge f =e-GC(e(1-f))$, we get that $GC(f(1-e))=f-e\wedge f$. Hence $GC(fe)=f-(1-e)\wedge f$ that is $ e' =f-(1-e)\wedge f$. As $e'\sim f'$ we get that $f'\sim e'$. Therefore  $e-e\wedge (1-f) \sim  f-(1-e)\wedge f$.  Replacing $f$ by $1-f$, we get that $e-e\wedge f\sim 1-f-(1-e)\wedge (1-f)$ which gives $e-e\wedge f\sim [1-(1-e)\wedge (1-f)]-f$ . Hence $e-e\wedge f\sim e\vee f-f$. Thus the parallelogram law holds in $R$.
\end{proof}	
Following is the result about orthogonality of two projections in a  generalized  p.q.-Baer  $*$-ring.

\begin{proposition}\label{s5pr4} If the projections $e$ and $f$ are very orthogonal  in a  generalized  p.q.-Baer  $*$-ring then; 
\begin{enumerate}
	\item $e , f$ are orthogonal;
	\item $GC(e)GC(f) =0$;
	\item $eRf=0$.
\end{enumerate}
\end{proposition}
\begin{proof}
As 	$e$ and $f$ are very orthogonal, there exists a central projection $h$ such that $he=e$ and $hf=0$. Therefore $ef=hef=ehf=0$, that is $e , f$ are orthogonal.
Let us prove that $GC(e)GC(f) =0$.
We have $GC(e)GC(f)=GC(he)GC(f)=hGC(e)GC(f)=GC(e)GC(hf)=GC(e)GC(0)	=0$.
 Finally, $eRf=heRf=eRhf=0$, since $hf=0$ .
\end{proof}
In \cite{Ber}, the author proved that Baer  $*$-ring satisfies partial  comparability. 
The following result shows that generalized  p.q.-Baer  $*$-ring satisfies partial  comparability.
\begin{proposition}\label{s5pr5} Let $R$ be a generalized  p.q.-Baer  $*$-ring, and $x,y\in R$. If  $xRy\neq 0$, then there exists a central projection $h$ such that $h\leq GC(x), h\leq GC(y)$. Hence $R$ has $PC$.
\end{proposition}
\begin{proof}
	We have $xRy\neq 0$. Let $z=xay\in xRy$ such that $z\neq 0$.
	Suppose $GC(z)=h$. Therefore there exists $n\in \mathbb N$ such that $z^nh=z^n$. Now $zGC(y)=xayGC(y)=xay=z$. Hence $z^nGC(y)=z^n$ which implies $ GC(z)\leq GC(y)$. Thus $h\leq GC(y)$.
	Also, $GC(x)z=GC(x)\cdot xay=xay=z$. Hence $GC(x)z^n=z^n$ that is $ GC(z)\leq GC(x)$.
	Therefore $h\leq GC(x)$. Let us prove that $R$ has PC. Let $x=e$ and $y=f$. We have  $z=eaf\in eRf$ such that $z\neq 0$. Therefore $GC(z)=h$ is such that $h\leq e ,h\leq f$ and $h\sim h$, thus $R$ has $PC$. 
\end{proof}
In \cite{Ber}, Berberian proved the result regarding the orthogonal decomposition of projections in a Baer $*$-ring. The following result is analogous for generalized  p.q.-Baer  $*$-ring.
\begin{proposition}\label{s5pr6} Let $R$ be a generalized  p.q.-Baer  $*$-ring satisfying parallelogram law. If $e,f $ are projections in $R$, then there exists an orthogonal decomposition $e=e'+e'',f=f'+f''$ with $e'\sim f'$ and $ef''=fe''=0$ where both $e',f'$ are central projections.
\end{proposition}
\begin{proof}
	Let $GC(ef)=f' ,GC(fe)=e'$. Hence both $e' ,f'$ are central projections. Thus as in the above proposition \ref{s5pr5}, $GC(e'f')=f', GC(f'e')=e'$, and hence $e'$ and $f'$ are in position $p'$ which gives $e'\sim f'$. Let $e''=e-e' ,f''=f-f'$. Now $e''(ef)=(e-e')ef=ef-e'ef=ef-ef=0$. Therefore $(e''e)f=0$ this implies $e''f=0$ since $e''\leq e$. Similarly, $(ef)f''=0$ implies  $ ef''=0$.
\end{proof}	 
 
\section{A Separation Theorem for Generalized P.Q.-Baer $*$-Rings}

 In this section, we introduce the concept of generalized central strict ideals in a $*$-ring. Also, we prove a separation theorem for generalized p.q.-Baer $*$-rings.  
\begin{definition}  An ideal $I$ of a generalized  p.q.-Baer  $*$-ring $R$ is called a {\it generalized central strict ideal} if $x\in I$ implies $GC(x)\in I$.
Further, a proper generalized central strict ideal $P$ is called {\it prime generalized central strict ideal}, if $IJ \subseteq P $ then either $I \subseteq P $ or $J \subseteq P $ for generalized central strict ideals $I$ and $J$. 
\end{definition}
The following result gives a relation between generalized central strict ideals and generalized central covers  in generalized p.q.-Baer $*$-rings.
\begin{lemma} \label{s6lm1} Let $R$ be  a generalized  p.q.-Baer  $*$-ring.
\begin{enumerate}
	\item If $e$ is a central projection in $R$ then $I=<e>$ is the generalized central strict ideal;
	\item For a generalized central strict ideal $I$, if $GC(x)\in\ I$ then $x^n\in I$ for some $n\in \mathbb N$.
	\end{enumerate}
\end{lemma}
\begin{proof}
$(1):$ We  have $I=<e>=\{xe~|~x\in R\}$.
Let $xe\in I$. We have $GC(xe) = eGC(x) \in I$ since $e\in R$ and $GC(x)\in I$.
Hence if $xe\in I$ then $ GC(xe) \in I$.\\
$(2):$ Let $GC(x) = e \in I$. Therefore $x^ne=x^n$. Hence if $ e \in I$ and $x^n \in R$ then $ x^ne \in I$ which gives $x^n \in I$. 
Therefore $ GC(x) \in I$ which gives $x^n \in I$. 
\end{proof}

 In the following  examples we have $x\notin I$ but $GC(x) \in I$, for a generalized central strict ideal $I$ of a ring $R$.
   \begin {example} Let $R=M_{2}(\mathbb{C})$ with conjugate transpose as an involution.
   Let $J=\{zI_{2}|z\in \mathbb{C}\}$. Note that $0$ and $I_{2}$ are only central projections in $R$. Also, $J$ is a generalized central strict ideal because for any $x\in J, GC(x)\in J$ as $x^ne=x^n$ is satisfied by $e=I$ only. Let $x=\begin {bmatrix} 
   1& 0\\0&0
   \end{bmatrix} \in R $. We have $GC(x) = I_{2}\in J$ but $x\notin J$. Observe that $x^ne=x^n$ is not satisfied by $e=0$ hence $GC(x) = I_{2}$.
   
\end{example} 
\begin {example}  Let $R=M_{2}(\mathbb Z_{2})$ with transpose as an involution. Only central projections in $R$ are $0$ and $I_{2}$. Let $I=\{aI_{2}~|~a\in \mathbb Z_{2}\}=\{0,I_{2}\}$. Let $x=\begin {bmatrix} 
1& 0\\0&0
\end{bmatrix}\in R$. Clearly, $x\notin I$. Note that the smallest central projection $e$ such that $x^ne=x^n$ is $I_{2}$. Thus $GC(x)=I_{2} \in I$.

\end{example}
Note that if $I$ is a prime generalized central strict ideal of $R$ and $e \in R$ is a central projection then $I$ contain exactly one of $e$ or $1-e$. 
\begin{lemma} \label{s6lm2} Let $R$ be a generalized   p.q.-Baer  $*$-ring and $P$ be a generalized central strict ideal of $R$. If $P$  is prime in a usual sense then $P$ is a prime generalized central strict ideal. 
\end{lemma}
\begin{proof} 
	Let $IJ \subseteq P, $ where $I,J$ be a generalized central strict ideals of $R$.  We  prove that $I \subseteq P $ or $ J \subseteq P$. 
	Suppose $ I\nsubseteq P$. We claim that $J\subseteq P$. Let $b_j\in J$. Choose $a_i\in I$ with $a_i\notin P$.
	Thus $a_ib_j\in IJ$ which means $a_ib_j\in P$. As $P$ is prime  and $a_i\notin P$. So $b_j\in P$ and hence $J\subseteq P$. Therefore $P$ is a prime generalized central strict ideal.
\end{proof}

  An ideal $I$ of $*$-ring $R$ is said to be {\it restricted} if $ I=\langle \widetilde{I} \rangle$, that is an ideal generated by projections in $I$ coincide with $I$.

 Now we give an example of a generalized   p.q.-Baer  $*$-ring in which a generalized central strict ideal is not restricted. 
 \begin {example} Let $R=M_2(\mathbb Z_{4})$. Let $I=\left\{\begin {bmatrix} 
 2a& 2b\\2c& 2d
 \end{bmatrix}|~a,b,c,d\in \mathbb Z_{4}\right\} $. Observe that for any $x\in I ,~x^n=0$ for some $n\in \mathbb N$. Therefore $x^ne=x^n =0$ and hence $GC(x)=0\in I$. So $I$ is generalized central strict ideal. The only projection in $I$ is a zero matrix. Thus $<\widetilde{I}> = <0>=\{0\}\neq I$. Hence $I$ is not a restricted ideal.
 \end{example}
 Let $B(R)$ be the Boolean algebra of central projections of $*$-ring $R$. For any $e,f \in B(R),~ e\vee f=e+f-ef$ and $e\wedge f=ef$ are in $B(R)$.
 
 \begin{lemma}\label{s6lm3} Let $R$ be a generalized  p.q.-Baer  $*$-ring and $b,c\in R$. Then  $GC(b^n+c^m)\leq GC(b)\vee GC(c)$ and $GC(b^nc^m)\leq GC(b)\wedge GC(c)$ for some $m,n\in \mathbb{ N}$.
 \end{lemma}
 \begin{proof}
 	Let $GC(b) =f$ and $GC(c)=g$. Therefore there exists $m,n\in \mathbb{ N}$ such that $b^nf=b^n$ and $c^mg=c^m$. Let $a=b^n+c^m$ and $GC(a) =e.$  Thus $a^ke=a^k$. Therefore $a(f\vee g)=(b^n+c^m)(f+g-fg)=b^nf+b^ng-b^nfg+c^mf+c^mg-c^mfg =b^n+b^ng-b^ng+c^mf+c^m-c^mf=b^n+c^m=a$. Hence $a(1-(f\vee g))=0$ which gives $(aR)^k(1-(f\vee g))=0$. This implies $e(1-f\vee g)=0$. So $e\leq f\vee g$. Therefore $GC(b^n+c^m)\leq GC(b)\vee GC(c)$.
 	 Now we  prove that $GC(b^nc^m)\leq GC(b)\wedge GC(c)$ for some $m,n\in \mathbb{ N}$. Let $GC(b^nc^m)=h$. Now $(b^nc^m)^kh=(b^nc^m)^k$ for some $k\in \mathbb{N}$. We have  $GC(b) =f$ and $GC(c)=g$. As $b^nc^m(1-g)=b^nc^m-b^nc^mg=b^nc^m-b^nc^m=0$. We get that $(b^nc^mR)^k(1-g)=0$ which gives $h(1-g)=0$ i.e. $h \leq g$. Similarly $h\leq f$. Hence $h\leq f\wedge g$. Thus $GC(b^nc^m)\leq GC(b)\wedge GC(c)$ for some $m,n\in \mathbb{ N}$.
 \end{proof}
 \begin{corollary}\label{s6cr1} Let $R$ be  a generalized  p.q.-Baer  $*$-ring and $a_1,a_2,\cdots, a_k\in R$ then there exists $n_1,n_2,\cdots, n_k\in \mathbb{N} $ such that $GC(a_1^{n_1}+a_2^{n_2}+\cdots +a_k^{n_k})\leq GC(a_1)\vee GC(a_2)\vee \cdots\vee GC(a_k) $ and   $GC(a_1^{n_1}a_2^{n_2}\cdots a_k^{n_k})\leq GC(a_1)\wedge GC(a_2)\wedge \cdots\wedge GC(a_k) $.	
 \end{corollary}
 Next, we give a characterization  for  prime generalized central strict ideal in a generalized p.q.-Baer  $*$-ring.
 \begin{theorem} \label{s6tm1} Let $R$ be  a generalized  p.q.-Baer  $*$-ring and $Q$ be a generalized central strict ideal of $R$. Then $Q$ is a prime generalized central strict ideal if and only if $Q \bigcap B(R)$ is maximal in $B(R)$.
 \end{theorem}
 \begin{proof} Let $Q$ be a prime generalized central strict ideal in $R$. Therefore $Q$ is proper in $R$ and hence $Q \bigcap B(R)$ is proper in $B(R)$. We have to prove $Q \bigcap B(R)$ is maximal in $B(R)$. Let $P$ be a maximal ideal in $B(R)$ such that $Q \bigcap B(R) \subseteq P$. We claim that $Q \bigcap B(R)=P$. Suppose $Q \bigcap B(R)\subset P$. Therefore there exists a central projection $e$ such that $e\in P$ but $e\notin Q$. As $Q$ is  a prime generalized central strict ideal, we have $1-e\in Q$ which gives $1-e\in P$, which is a contradiction since $P$ is a maximal and $e,1-e\in P$. Hence $Q \bigcap B(R)=P$  and $Q \bigcap B(R)$ is maximal in $B(R)$.  Conversely, suppose $Q \bigcap B(R)$ is maximal in $B(R)$. We prove that $Q$ is a prime generalized central strict ideal of $B(R)$. Let $I,J$ be a generalized central strict ideals such that $IJ\subseteq Q$. Suppose $I\nsubseteq Q$ and $J\nsubseteq Q$. Therefore there exists $x\in I-Q$ and $y\in J-Q$. As $I,J,Q$ are generalized central strict ideals, we get $e=GC(x)\in I $ and $f=GC(y)\in J$. Hence $ef\in Q \bigcap B(R)$. We have A $ \bigcap B(R)$ is maximal hence it is  prime. Therefore $e\in Q \bigcap B(R)$ or $f\in Q \bigcap B(R)$ which is a contradiction because $e \notin Q$ and $f \notin Q$. 
 	Hence either $I\subseteq Q$ or $ J \subseteq Q$. Thus $Q$ is a prime generalized central strict ideal. 
\end{proof}
 \begin{definition} A non-empty set $S$ of a generalized  p.q.-Baer  $*$-ring $R$ is said to be a $GC$- set if the following conditions hold. 
 \begin{enumerate}
 \item  $0 \notin S$.
 \item If $x \in S$ then $GC (x)\in S$.
 \item $S$ is multiplicatively closed.
 \end{enumerate} 
 \end{definition} 
 In the following result, we give a condition so that an ideal becomes a prime generalized central strict ideal.
 \begin{theorem}\label{s6tm2} Let $P$ be an ideal of a generalized  p.q.-Baer  $*$-ring $R$. If $R-P$ is a $GC$-set such that $GC(x) \in R-P$ implies $x\in R-P$ then $P$ is a prime generalized central strict ideal
 \end{theorem}
 \begin{proof}
 	 	Clearly, $P$ is proper. If not  then $R-P=\phi$. But $R-P$ is a $GC$-set and so it is not empty. Suppose $R-P$ is a $GC$-set. Therefore by definition and assumption $x \in R-P$ if and only if $GC(x) \in R-P$. Thus $x\in P$ if and only if  $GC(x)\in P$. Therefore $P$ is a generalized central strict ideal. We prove that $P$ is a prime generalized  central strict ideal. Let $I,J$ be generalized central strict ideals such that $IJ \subseteq P$. If $I \nsubseteq P$ and $J\nsubseteq P$ then there exists $x\in I-P,y\in J-P$. Hence $xy\in IJ\subseteq P$. Also, $x \in R-P, y\in R-P$ gives $xy \in R-P$ because $R-P$ is a $GC$ - set. Thus $xy \notin P$ which is a contradiction to $xy \in P$. Hence  either $I \subseteq P$ or $J \subseteq P$. Therefore $P$ is a prime generalized central strict ideal.
  \end{proof}
 The following lemma gives a way to obtain  generalized central strict ideals. 
 	\begin{lemma}\label{s6lm4}	If $I$ and $J$ are  generalized central strict ideals of a generalized  p.q.-Baer  $*$-ring $R$, then $I^n+J^m$ and $I^nJ^m$ are generalized central strict ideals for some $m,n\in \mathbb{N}$.
 	\end{lemma}
 \begin{proof}
 	 Let $b\in I$ and $c\in J$. Let $GC(b)=f$ and $GC(c)=g$. Therefore there exists $m,n\in \mathbb{N} $ such that $b^nf=b^n , c^mg=c^m$. Now $b^n+c^m\in I^n+J^m$. Suppose $GC(b^n+c^m)=e$. We have $GC(b^n+c^m)\leq 	GC(b)\vee GC(c)$. Therefore $e\leq f\wedge g$ which gives $e=e(f\vee g)$. Hence $e=e(f+g-fg)=ef+eg-efg$. As $b\in I$, we get that $GC(b)\in I$ that is  $f\in I$ implies $f\in I^n$.  Similarly, $g\in J^m$. Therefore $ef\in I^n$ and $(e-ef)g\in J^m$.  Thus $e=ef+eg-efg\in I^n+J^m$. Hence $I^n+J^m$ is generalized central strict ideals for some $m,n\in \mathbb{N}$.  Now we prove that $I^nJ^m$ is a  generalized central strict ideal for some $m,n\in \mathbb{N}$. Let $b^nc^m\in I^nJ^m$ where $b\in I,c\in J$. Let $GC(b^nc^m)=e , GC(b)=f, GC(c)=g$. Therefore $f\in I, g\in J$. We have $GC(b^nc^m) \leq GC(b)\wedge GC(c)$. Thus $e\leq f\wedge g$ i.e. $e\leq fg$. Hence $e=efg$. As $ef\in I^n ,g\in J^m$ implies $efg\in I^nJ^m$ i.e. $e\in I^nJ^m$. Suppose $b_1^nc_1^m +b_2^nc_2^m+\cdots +  b_k^nc_k^m\in I^nJ^m$ where   $b_i\in I,c_i\in J$. Let $GC(b_i)=f_i, GC(c_i)=g_i, GC(b_1^nc_1^m +b_2^nc_2^m+\cdots +  b_k^nc_k^m)=e$. As
 	 $ GC(b_1^nc_1^m +b_2^nc_2^m+\cdots +  b_k^nc_k^m)\leq GC(b_1c_1)\vee \cdots \vee GC(b_kc_k)\leq (GC(b_1)\wedge GC(c_1))\vee \cdots\vee (GC(b_k)\wedge GC(c_k))$. Thus $e\leq (f_1\wedge g_1)\vee (f_2\wedge g_2)\vee \cdots\vee  (f_k\wedge g_k)\leq (f_1 g_1)\vee (f_2 g_2)\vee \cdots\vee  (f_k g_k)$. So $e=e[(f_1 g_1)\vee (f_2 g_2)\vee \cdots\vee  (f_k g_k)]= \displaystyle \sum_{i=1}^{k}(ef_i)g_i -\sum _{1\leq i<j\leq k}(ef_if_j)(g_ig_j)+\cdots +
 	  (ef_1\cdots f_k)(g_1\cdots g_k)$. Observe that each term on right hand side belongs to $I^nJ^m$. Hence $e\in I^nJ^m$.  Thus $I^nJ^m$ is a generalized central strict ideal for some $m,n\in \mathbb{N}$.
 	 \end{proof}
 	 Now, we prove a separation theorem for a generalized  p.q.-Baer  $*$-ring.
 	 \begin{theorem}\label{s6tm3} Let $R$ be a generalized  p.q.-Baer  $*$-ring and $I$ be a generalized central strict ideal of $R$. Let $M$ be a $GC$-set of $R$ with $M\bigcap I=\phi$. Then there exists a prime generalized central strict ideal of $R$ containing $I$ disjoint from $M$.
 	\end{theorem}
 	\begin{proof} By Zorn's lemma, let $Q$ be a maximal generalized central strict ideal such that $I\subseteq Q$ with  $Q\bigcap M=\phi$. We  prove that $Q$ is a prime ideal. Let $J_1\nsubseteq Q$ and $J_2\nsubseteq Q$. Now there exists $m,n,k,l\in \mathbb{N}$ such that $J_1^m+Q^k$ and $J_2^n+Q^l$ are generalized central strict ideals. We have $Q^k\subseteq J_1^m+Q^k$ and $Q^l\subseteq J_2^n+Q^l$. Let $\alpha =max\{k,l\}$. As $I\subseteq Q$. Thus $I^\alpha\subseteq Q^\alpha\subseteq Q^k\subseteq J_1^m+Q^k$ and $I^\alpha\subseteq  J_2^n+Q^l$. Also, $Q$ is a maximal ideal containing $I$ with $Q\bigcap M=\phi$. Therefore $Q$ is also a maximal ideal containing $I^\alpha $ with $Q\bigcap M=\phi$. Therefore  $(J_1^m+Q^k)\bigcap M\neq\phi,  (J_2^n+Q^l)\bigcap M\neq\phi $. Let $s=\sum (j_1\cdots  j_m)+\sum(q_1\cdots q_k), t=\sum j_1^{\prime}\cdots j_n^{\prime}+\sum q_1^{\prime}\cdots q_l^{\prime}$ be elements  in $(J_1^m+Q^k)\bigcap M$ and $(J_2^n+Q^l)\bigcap M$ respectively. If $J_1J_2\subseteq Q$ then  $st\in Q$. As $M$ is $GC$-set, $st\in M$. Thus $st\in M\bigcap Q$ which is a contradiction to $Q\bigcap M=\phi$. Therefore 
 	  $J_1J_2\nsubseteq Q$. Hence $Q$ is a  maximal generalized prime central strict ideal of $R$. 
 \end{proof}
 
  \section{Sheaf Representation of Generalized  P.Q.-Baer  $*$-Ring} 	 
 
 In \cite{Shin},  Shin obtained a sheaf representation of pseudo symmetric rings. In \cite{ahmadiquasi2020, Bir1}, authors have given a sheaf representation for generalized quasi-Baer $*$-rings and quasi-Baer rings respectively. In this section, we give a sheaf representation of generalized p.q.-Baer $*$-rings.
 
 Let $R$ be a generalized  p.q.-Baer  $*$-ring. Let $\sum(R)$ denote  the collection of prime generalized central strict ideals in $R$. The kernel of $A\subseteq \sum(R)$ is defined as $K(A)=\displaystyle \bigcap_{Q\in A}{Q}$. Let $S\subseteq R$. The hull of $S$ is defined as $H(S)=\{Q\in \sum (R)~|~S\subseteq Q\}$ is the set of prime generalized central strict ideals in $R$ containing $S$. 
 Let $P(S)= \{Q\in \sum (R)~|~S\nsubseteq Q\}$. If $S=\{x\}$ then we write 
 $P(S)=P(x)$ and $H(S)=H(x)$.  Let $x\in R,~GC(x)=e,Q\in \sum (R)$. We have \\ 
 $(1)$ $P(0)=H(1)=H(R)=\phi$.\\ 
 $(2)$ $P(1)=H(0)=P(R)=\sum (R)$. \\
 $(3)$ Now $x\in Q$ implies $e\in Q$. Therefore $Q\in H(x)$ which implies $Q\in H(e)$. Hence $H(x)\subseteq H(e) $. Therefore $e\notin Q$ that is $ x\notin Q$ implies $ Q\in P(e)$ that is $ Q\in P(x)$. Hence $P(e)\subseteq P(x) $. Also, $e\in Q$ implies $ x^n\in Q$. That is  $Q\in H(e)$ implies $ Q\in H(x^n)$ for some $n\in \mathbb{N}$. Thus $H(e)\subseteq H(x^n) $  for some $n\in \mathbb{N}$. Further, if $x^n\notin Q$ then $ e\notin Q$. Hence $Q\in P(x^n)$ implies $Q\in P(e)$. Therefore $P(x^n)\subseteq P(e) $. Finally combining these results we have
 $H(x)\subseteq H(e)\subseteq H(x^n)$ and $P(x^n)\subseteq P(e)\subseteq P(x) $.\\
 $(4)$ We have $Q\in P(x)$ if and only if $x\notin Q$ if and only if   $<x>\nsubseteq Q$ if and only if $Q\in P(<x>)$. Thus $P(x)=P(<x>)$. \\Similarly, $H(x)=H(<x>)$. \\
 $(5)$ For $A,B\in \sum (R),~ P(A\bigcap B)=P(A)\bigcap P(B)=P(AB)$ and $H(A\bigcap B)=H(A)\bigcup H(B)=H(AB)$.
 
 The following result gives a basis for open sets in $\sum (R)$.
 
 \begin{proposition}\label{s7pr1} In a generalized  p.q.-Baer $*$-ring $R$ the set $\{P(x)~|~x\in R\} $ forms a basis for open sets in $\sum (R)$.
 \end{proposition}
 \begin{proof} Note that if $Q\in \sum (R)$ then there exists $x\in R$ such that $x\notin Q$.
 	Thus $Q\in P(x)$. Let $Q\in P(x_1)\bigcap P(x_2)$. Suppose $GC(x_1)=e_1,~	GC(x_2)=e_2$. So $x_1^{n_1}e_1=x_1^{n_1},x_2^{n_2}e_2=x_2^{n_2}$. To prove $Q\in P(e_1e_2)\subset P(x_1)\bigcap P(x_2)$. We have $x_1\notin Q,~ x_2\notin Q$. Thus $x_1^{n_1}\notin Q, x_2^{n_2}\notin Q$. Hence $Q\in P(x_1^{n_1}), Q\in P(x_2^{n_2})$. But $P(x^n)\subseteq P(e) $. Therefore  $Q\in P(e_1) ,Q\in P(e_2)$.  That is $e_1\notin Q, e_2\notin Q$. Thus $<e_1><e_2>\nsubseteq Q$ implies $<e_1e_2>\nsubseteq Q$ because $e_1 ,e_2$ are central. Hence $e_1e_2\notin Q$  implies $Q\in P(e_1e_2)$. Let us prove that $ P(e_1e_2)\subseteq P(x_1)\bigcap P(x_2)$. Let $Q^{\prime} \in P(e_1e_2)$.  Therefore $e_1e_2\notin Q^{\prime}$. Thus $e_1\notin Q^{\prime}$ since if $e_1\in Q^{\prime}$ then $e_1e_2\in Q^{\prime}$ which is a contradiction. Similarly, $e_2\in Q^{\prime}$. Hence $Q^{\prime} \in P(e_1), Q^{\prime} \in P(e_2)$. Therefore $Q^{\prime} \in P(x_1), Q^{\prime} \in P(x_2)$ because $P(e)\subseteq P(x) $. Therefore $Q^{\prime}\in P(x_1)\bigcap P(x_2)$. Hence  $Q\in P(e_1e_2)\subseteq P(x_1)\bigcap P(x_2)$. Thus $\{P(x)~|~x\in R\} $ forms a basis for open sets in $\sum (R)$. This basis defines a topology on $\sum (R)$, called Hull-Kernel topology.
 	\end{proof}
 	
  If $S$ is a subset of a ring $R$ then $Ann(S)=\{x\in R~|~sx=0 ~ or ~  xs=0\}$. 		The following proposition is required to develop the theory of Hull and Kernel.
 	\begin{proposition}\label{s7pr2} Let $R$ be a generalized  p.q.-Baer  $*$-ring and $I$ be an ideal of $R$. Then $Ann(I)$ is a	generalized central strict ideal of $R$.
 		
 \end{proposition} 
 \begin{proof} Clearly for $a\in R ,~ aI=\{0\}$ if and only if $aRI=\{0\}$ if and only if $IRa=\{0\}$ if and only if $Ia=\{0\}$. Therefore $l(I)=r(I)=Ann(I)$. Thus $Ann(I)$ is an ideal. We  prove that $Ann(I)$ is the central strict ideal of $R$. 
 	Let $x\in Ann(I)$  and $GC(x)=e$. Therefore there exists $n\in \mathbb{N}$ such that $x^ne=x^n , (xR)^ny=\{0\}  $ implies $ ey=0$. This gives $xI=0$ implies $xRI=\{0\}$ that is $ (xR)^nI=\{0\}$ that is $ eI=0 $. Hence $ e\in Ann(I)$. Thus $Ann(S)$ is a	generalized central strict ideal of $R$.
  \end{proof}	
 	In the following successive two results, we express $Ann(I)$ in terms of Hull and Kernel, where $I$ is a generalized central strict ideal of  generalized  p.q.-Baer  $*$-ring. 
  \begin{theorem}\label{s7tm1} If  $I$ is a generalized central strict ideal of  generalized  p.q.-Baer  $*$-ring $R$ then $Ann(I)=K(\sum (R) -H(I))$.	
 \end{theorem}
 \begin{proof} By Proposition \ref{s7pr2}, $Ann(I)$ is a	generalized central strict ideal of $R$. Let $x\in Ann(I)$. Thus $Ix=0$. Now for any $Q$ with $I\nsubseteq Q$ there exists $y\in I$ such that $y\notin Q$. As $yx=0$. Therefore $yx\in Q$. So $x\in Q$ as $y\notin Q$. Hence $x\in Q$ for all $Q$ with 	$I\nsubseteq Q$. Thus $x\in \bigcap \{Q\in \sum(R)~|~I\nsubseteq Q \}$ which implies $Ann(I) \subseteq \bigcap \{Q\in \sum(R)~|~I\nsubseteq Q \}$. Let us prove that $\bigcap \{Q\in \sum(R)~|~I\nsubseteq Q \} \subseteq Ann(I)$. Suppose that $x\in \bigcap \{Q\in \sum(R)~|~I\nsubseteq Q \} $ but $x\notin Ann(I)$. Thus there exists $y\in I$ such that $xy\neq 0$. Now if $I\nsubseteq Q$ then $x\in Q$ and hence $xy\in Q$. If $I\subseteq Q $ then $y\in Q$ which gives $xy\in Q$. Therefore $xy\in \bigcap \{ Q~|~ Q\in \sum(R) \}=\{0\}$ which implies $xy=0$ which is a contradiction. Thus $x\in Ann(I)$. Hence $Ann(I)=\bigcap \{Q\in \sum(R)~|~I\nsubseteq Q \}$. We have
 	 $\bigcap \{Q\in \sum(R)~|~I\nsubseteq Q \} =\bigcap \{Q\in \sum(R)~|~Q\notin H(I) \}=\bigcap \{Q\in \sum(R)~|~Q\in \sum(R)-H(I) \} =K(\sum(R)-H(I))$. Thus $Ann(I)=K(\sum (R) -H(I))$.
 \end{proof}

 \begin{theorem}\label{s7tm21} Let $I$ be a generalized central strict ideal of a generalized  p.q.-Baer  $*$-ring $R$ then $I=K(H(I))$.
 \end{theorem}
 \begin{proof} We have for any $GC$-set $M$ with $M\bigcap I= \phi$, there exists generalized prime central strict ideal $Q$ such that $I\subseteq Q,$ and $ Q\bigcap M=\phi $. Hence $H(I) \neq \phi$. Thus by definition of $H$ and $K$  we have $I\subseteq K(H(I))$. Let us prove that $K(H(I)) \subseteq I$. Let $x\in K(H(I))$ and $x\notin I$. Suppose  $GC(x)=e$. We have $e\notin I$. As $x\in K(H(I))$, so $e\in K(H(I))$. Take $F=\{e\}$ which is a $GC$-set such that $F\bigcap I= \phi$. By separation theorem there exists a generalized prime central strict ideal $Q'$ such that $I\subseteq Q' , Q'\bigcap F=\phi $. Therefore $e\notin Q'$ and $I\subseteq Q'$. Thus $e\notin K(H(I))$, a contradiction. So $K(H(I)) \subseteq I$ and hence $I=K(H(I))$.
 \end{proof}	
 \begin{corollary}\label{s7cr1}	 Let $R$ be a generalized  p.q.-Baer  $*$-ring and $a,b\in R$ then $<GC(a)>\bigcap <GC(b)>  =  <GC(a)GC(b)>$. 
 \end{corollary}
 \begin{proof} First we prove that  $H(<GC(a)>\bigcap <GC(b)>) =H(<GC(a)GC(b)>)$. Let $Q\in H(<GC(a)>\bigcap <GC(b)>) $. Therefore   $<GC(a)>\bigcap <GC(b)> \subseteq Q$. Hence $<GC(a)GC(b)> =<GC(a)><GC(b)>\subseteq <GC(a)>\bigcap <GC(b)>\subseteq Q$. Therefore $Q\in H(<GC(a)GC(b)>) $. Thus $H(<GC(a)>\bigcap <GC(b)>) \subseteq H(<GC(a)GC(b)>)$.	Let $Q'\in H(<GC(a)GC(b)>) $. This gives $<GC(a)GC(b)>\subseteq Q'$, that is $<GC(a)><GC(b)>\subseteq Q'$. As $Q'$ is a generalized prime central strict ideal, we have $<GC(a)>\subseteq Q'$ or $<GC(b)>\subseteq Q'$. Therefore  $<GC(a)>\bigcap <GC(b)> \subseteq Q'$. Hence $Q' \in H(<GC(a)>\bigcap <GC(b)>)$.  Therefore $H(<GC(a)>\bigcap <GC(b)>)=H(<GC(a)GC(b)>)$. By Theorem  \ref{s7tm21},  $<GC(a)>\bigcap <GC(b)> = <GC(a)GC(b)>$. 
 \end{proof}	
 According to Hong et al. \cite{Hong}, a ring $R$ is said to have the {\it right annihilator condition} (or right a.c.), if for any $a,b\in R$ there exists $c\in R$,  such that $r(aR)\bigcap r(bR) =r(cR)$. Rings with left a.c. are defined similarly. A ring $R$ is said to have a.c. if $R$ has left and right a.c..
 \begin{definition} A ring $R$ is said to have a {\it generalized  right annihilator condition} if for any $a,b\in R$ there exists $c\in R$ and $m,n\in \mathbb{N}$  such that $r(aR)^n \bigcap r(bR)^m =r(cR)$. Similarly, a ring with a {\it generalized  left  annihilator condition} is defined. A ring $R$ is said to have a {\it  generalized a.c.} if $R$ has generalized left a.c. and generalized right a.c..
 \end{definition}
 In \cite{Hong}, authors proved a result regarding the annihilator condition for  rings. 
 Next result shows that generalized p.q.-Baer $*$-rings have the right annihilator condition.
 \begin{lemma}\label{s7lm1} Let $R$ be a  generalized  p.q.-Baer  $*$-ring then it has a generalized  right annihilator condition.
 \end{lemma}	
 \begin{proof} Let $x,y\in R$. As $R$ is {\it generalized  p.q.-Baer  $*$-ring}, there exists $m,n\in \mathbb{N}$  such that $r(xR)^n=(1-GC(x))R$ and $r(yR)^m=(1-GC(y))R$. We have $1-GC(x) , 1-GC(y) $ are central projections, therefore $GC(1-GC(x))=1-GC(x)$ and $GC(1-GC(y))=1-GC(y)$. Thus by corollary \ref{s7cr1}, $<1-GC(x)>\bigcap <1-GC(y)> = <(1-GC(x))(1-GC(y))> =<1-( GC(x)+GC(y)-GC(x)GC(y))>= <1-e>$ where $e=GC(x)+GC(y)-GC(x)GC(y)$. Hence $r(xR)^n \bigcap r(yR)^m =r(eR)$.
 \end{proof}	
 
 It is known that for a commutative ring $R$ with unity, the $Spec(R)$ is a $T_0$ space. 
 In the following result, we prove that if $R$ is a generalized  p.q.-Baer  $*$-ring then $\sum(R) $ with Hull-Kernel topology is a $T_2$ space. 
 \begin{theorem}\label{s7tm3}If $R$ is  generalized  p.q.-Baer  $*$-ring then $\sum(R) $ with Hull-Kernel topology is a zero-dimensional Hausdorff space. 	
 \end{theorem}
 \begin{proof} For any $x\in R $ and  $Q\in \sum(R)$ we have $Q\in P(x)$ if and only if $GC(x)\notin Q$ if and only if $1-GC(x)\in Q$ if and only if $Q\notin P(1-GC(x))$ if and only if $Q\in \sum(R)-P(1-GC(x))$. This gives $P(x)= \sum(R)-P(1-GC(x))$. Thus basis elements of Hull-Kernel topology on $\sum(R)$ are open and closed sets. Thus $\sum(R)$ with Hull-Kernel topology is zero-dimensional space.  To prove this is  Hausdorff space. Let $A ,B$ be distinct generalized  prime central strict ideals of $R$. Take $x \in A-B$. Therefore $x \in A$ and $x \notin B$. Hence $B\in P(x)=P(GC(x))$ and $A\in P(1-GC(x))$. So $P(GC(x))$ and $P(1-GC(x))$ are disjoint open sets containing $B$ and $A$ respectively. Hence $\sum(R) $ with Hull-Kernel topology is a zero-dimensional Hausdorff space.
 \end{proof}	
 In general, $MinSpec(R)$ is not compact even if $R$is a commutative reduced ring with unity. However, the space $\sum(R) $ with the hull-kernel topology is compact, when $R$ is a p.q.-Baer $*$-ring even though the elements of $\sum(R) $ are incomparable. The compactness of space with all elements incomparable plays a special role in the case of rings of continuous functions (see \cite{Hen}). 
 The following is an analogous result for a generalized p.q.-Baer $*$-rings.
 
 \begin{theorem}\label{s7tm4}If $R$ is generalized  p.q.-Baer  $*$-ring then; 
 	\begin{enumerate}
 		\item $\sum(R) $ with Hull-Kernel topology is  compact;
 		\item  The map $\phi :  \sum(R) \rightarrow \sum(B(R)),~\phi(Q)=Q\bigcap B(R)$ is a continuous bijection. 
 	\end{enumerate}	 
 \end{theorem}
 \begin{proof} $(1):$ Suppose $V$ is an open set in $\sum(R) $. So $V=\bigcup _{i\in J}P(x_i)=\bigcup _{i\in J}P(e_i)$ where $e_i=C(x_i)$. Let $I$ be the ideal generated by $\{e_i~|~i\in J\}$. We prove that $I$ is a central strict ideal such that $V=P(I)$. If $x\in I$, then $x=\sum e_ir_i$ for some $r_i\in R$ with $r_i=0$ except finitely many $i$. We have $GC(x)=GC(x)(\vee (e_iGC(r_i)))\in I$. So $I$ is a central strict ideal of $R$. If $Q\in V$ , then $Q\in P(x_i)$ for some $i\in J$ and hence $e_i\notin Q$. Consequently $I\nsubseteq Q$, and thus $Q\in P(I)$. Let us prove that $P(I)\subseteq V $. If  $Q'\in P(I)$ then  $I\nsubseteq Q'$. Therefore there exists $e_i\in I$ such that $e_i\notin Q'$ that is $Q'\in P(e_i)\subseteq V$. Hence $V=P(I)$. To prove the compactness of $\sum(R) $, it is sufficient to prove that every open cover of $\sum(R) $ having open sets of the form $P(A_i)$, where $A_i$ are generalized   central strict ideals that admits a finite sub cover. Let $\bigcup \{P(A_i)/i\in \textit{I}\}$ be an open cover of $\sum(R) $ where $\textit{I}$ is an indexing set and $A_i$ are generalized   central strict ideals of $R$. Then $\sum(R)= P(\sum _{i\in \textit{I}} A_i) $, where $\sum _{i\in \textit{I}} A_i $ is the ideal  sum of the generalized   central strict ideals $A_i $. Now $\sum _{i\in \textit{I}} A_i \nsubseteq Q$ for any $Q\in \sum(R) $. Thus $\sum _{i\in \textit{I}} A_i=R$. So $\sum a_i=1  , a_i\in A_i$, and only finitely many $a_i\neq 0$. Let $\textit{F}$ be the set of $i\in \textit{I}$ for which $a_i\neq 0$. Then $\sum _{i\in \textit{F}} a_i=1$. Consequently $\sum(R)=P(\sum _{i\in \textit{F}} A_i) $. Hence $\{P(A_i)/i\in \textit{F}\}$ is a finite subcover of $\{P(A_i)~|~i\in \textit{I}\}$. Hence $\sum(R) $ with Hull-Kernel topology is  compact.\\ 
 	$(2):$ Define  $\psi : \sum(B(R))\rightarrow \sum(R)$	  by $\psi (M)=MR$, where $M\in \sum(B(R)$. Clearly, the map is well defined. Let $P\in \sum(R)$, then $P$ is a minimal prime generalized   central strict ideal of $R$. Since  $\phi(P)=P\bigcap B(R)\in \sum(B(R)$, therefore $\psi(\phi(P))=\psi(P\bigcap B(R))=(P\bigcap B(R))R=P$ as $P$ is a restricted ideal. Let $M\in \sum(B(R))$, then $\phi (\psi(M))=\phi(MR))=MR\bigcap B(R)\in \sum(B(R))$. This gives $M\subseteq \phi(\psi(M))$. By maximality of $M$, we get $\phi (\psi(M))=M$. Thus $\phi$ and $\psi$ becomes inverses of each other. So $\phi$ and $\psi$ are bijections. To prove the continuity of $\phi$. Let $X$ be an open subset of $\sum(B(R))$.  Then $X=\{M\in \sum(B(R))/Y \nsubseteq M\}$ for some $Y\subseteq B(R)$. Hence $\phi^{-1}(X)=\{\phi^{-1}(M)\in \sum(R)/M\in X\}=\{Q\in\sum(R)/Y\nsubseteq Q\}$. Thus $\phi^{-1}(X)$ is an open set in $\sum(R)$. So $\phi $ is a continuous map. 
 \end{proof}		
 	
	The spectrum of a generalized  p.q.-Baer  $*$-ring can be used to a deep understanding of the sheaf Representation of such rings.
	\begin{definition} A triple $<S,\pi,X>$ consisting of two given sets $X$ and $S$ and $\pi:S\longrightarrow X$ a function of $S$ onto $X$, is called a {\it sheaf of $*$-rings} over the base set $X$ if the following conditions are satisfied.\\
		$(1)$ $\pi^{-1}(x)$ is a $*$-ring for each $x\in X$.\\
		$(2)$ $S$ and $X$ are topological spaces.\\ 
		$(3)$ Each point in $S$ has an open neighbourhood which is homeomorphically mapped onto an open set in $X$ under $\pi$.\\ $(4)$ The functions $(s,t) \rightarrow s+t$ and $(s,t) \rightarrow st$ from the set $S\Delta S=\{(x,y)\in S\times S~|~\pi (x)=\pi(y)\}$ into $S$ are continuous. Here $S\Delta S$ with the subspace topology induced from the product space $S\times S$.\\
		$(5)$ The function which assigns to every $x\in X$, the identity of $\pi^{-1}(x)$ is continuous.\\
		$(6)$ For every $x\in X$ ,the function $y\rightarrow y^*$ from $\pi^{-1}(x)$ to $S$ is continuous. If  for each $x\in X$, $\pi^{-1}(x)$ is a generalized  p.q.-Baer  $*$-ring then we say that $<S,\pi,X>$ is a sheaf of a generalized  p.q.-Baer  $*$-ring.
		
  \end{definition}	
  	\begin{definition} If $F=<S,\pi,X>$ is a sheaf of $*$-ring. For $x\in X$, the set $\pi^{-1}(x)$ is called a {\it stalk}. Any continuous function $\sigma :X\rightarrow S$ for which $\pi\circ \sigma$ is the identity map of $X$ is called a {\it section} over $X$. The set of all sections over $X$ is a $*$-ring, denoted by $\Gamma(X,S)$.
  \end{definition}	
  	If $R$ is a {\it generalized  p.q.-Baer  $*$-ring} and $r\in R$, then $\hat{r}:\sum(R)\rightarrow\displaystyle\bigcup_{Q\in \sum (R)}{R/Q}$ is the function defined by $\hat{r}(Q)=r+Q$ is a homomorphism. Let $\tau $ denote the hull-kernel topology on $\sum (R)$ defined as before. If $A\in \tau $ then $\hat{r}(A)=\{\hat{r}(Q)~|~Q\in A\}$.
  
  	We denote the set of all sections of the sheaf by $\Gamma(\sum (R),S)$. Observe that $\hat {r}\in\Gamma(\sum (R),S)$.
  	
  	\begin{lemma} \label{s7lm1} Let $R$ be a generalized  p.q.-Baer  $*$-ring and $x,y\in R$. Then $y\in Q$ for every $Q\in P(x)$ if and only if $GC(x)GC(y)=0$.
  		\end{lemma}
  	\begin{proof} Let $y\in Q$ for every $Q\in P(x)$. Therefore $GC(y)\in Q$	for every $Q\in P(x)$. That is  $GC(x)GC(y)\in Q$ for every $Q\in P(x)$. Also, if $Q\notin P(x)$ then $x\in Q$ implies $ GC(x)\in Q$ that is $ GC(x)GC(y)\in Q$. Hence $GC(x)GC(y)\in Q$ for any $Q\in \sum (R)$. As $\displaystyle\bigcap_{Q\in \sum (R)}{Q}=\{0\}$ we have $ GC(x)GC(y)=0$. Conversely, let $GC(x)GC(y)=0$. Suppose $Q\in P(x)$. Thus $x\notin Q$ and hence $x^n\notin Q$ for all $n$. Therefore $GC(x)\notin Q$.  Now $GC(x)GC(y)\in Q$ and $Q$ is prime generalized central strict ideal. Hence $GC(y)\in Q $ implies $ y^mGC(y)\in Q$ where $y^mGC(y)=y^m$. Therefore $ y^m\in Q $ that is $ y\in Q$ because $x\notin Q$ implies $x^n\notin Q$.
  \end{proof}
  Now, we determine the form of sections of the sheaf  and prove that all sections are injective.
  	\begin{theorem}\label{s7tm1 } Let $R$ be a  generalized  p.q.-Baer  $*$-ring and $F=<S,\pi,\sum (R)>$ be  a sheaf of $*$-rings as before. Then for every section $f$ of the sheaf $F$ over $S$ there exists $m\in\mathbb{N}$ such that $f^m$ has the form $\hat{r}$ for some $r\in R$. Moreover all sections are injective.  		
 \end{theorem}
  	\begin{proof} Let $f:\sum (R) \rightarrow S$ be any section of sheaf $F$. Therefore $f(Q)=r+Q=\hat{r}(Q)$ for any $Q\in \sum (R)$ and for some $r\in R$. As $f$ and $\hat{r}$ are continuous, there exists an open set $P(x)$ such that $Q\in P(x)$ and $f(A)=\hat{r}(A)$ for every $A\in P(x)$ for some $x\in R$. Since $\sum (R)$ is compact,  $\sum (R)= \displaystyle\bigcup_{i=1}^{n}P(x_i)$ for some $x_1, \cdots x_n\in R$. Hence $f(Q)=\hat{r_i}(Q)$ for all $Q\in P(x_i)$ for some  $r_1, \cdots r_n\in R$.	Write $GC(x_i)=e_i,GC(x_j)=e_j$. We have $P(x_i)\bigcap P(x_j) =P(<x_i>)\cap P(<x_j>)=P(e_ie_j)$. Hence if $Q\in P(x_i)\bigcap P(x_j)$ then $f(Q)=\hat{r_i}(Q)=\hat{r_j}(Q)$. Hence $r_i+Q=r_j+Q$ implies $r_i-r_j\in Q$. Thus $r_i-r_j\in Q$ for every $Q\in  P(e_ie_j)$. So $GC(r_i-r_j)GC(e_ie_j)=0$. Therefore  $GC(r_i-r_j)e_ie_j=0$ implies $((r_i-r_j)R)^me_ie_j=0$ for some $m\in \mathbb{N}$. Hence  $(r_i-r_j)^me_ie_j=0$. Now  $\sum (R)= \displaystyle\bigcup_{i=1}^{n}P(x_i)=P(\sum <e_i>)$. So $\sum <e_i>=R$ and hence 	$\displaystyle\sum _{i=1}^{n}e_iy_i=1$ for some $y_i\in R$. As $(r_i-r_j)^me_ie_j=0$ we get that $ e_i(r_i-r_j)^my_ie_j=0$ for $i=1,2\cdots$. Thus $e_i(r_i^m+g(r_i,r_j)+r_j^m)y_ie_j=0$. So $e_i(r_i^m+g(r_i,r_j)+r_j^m)y_i(x_jR)^n=0$ for some $n\in \mathbb{N}$. Hence $e_i(r_i^m+g(r_i,r_j)+r_j^m)y_ix_j^n=0$. Therefore \begin{equation} \label{eq1} e_ir_i^my_ix_j^m= e_i(g(r_i,r_j)+r_j^m)y_ix_j^n. \end{equation}  	  		
  	Let $a=	\displaystyle\sum _{i=1}^{n}e_ir_i^my_i$. Therefore for $x_j^n,\acute{r}\in R, a\acute{r}x_j^n= (\displaystyle\sum _{i=1}^{n}e_ir_i^my_i) \acute{r}x_j^n	=(\displaystyle\sum _{i=1}^{n}e_ir_i^my_i x_j^n)\acute{r}= \acute{r}	\displaystyle\sum _{i=1}^{n}e_i(g(r_i,r_j)+r_j^m)y_ix_j^n =\acute{r}\{(\displaystyle\sum _{i=1}^{n}e_iy_i)r_j^mx_j^n+\displaystyle\sum _{i=1}^{n}e_ig(r_i,r_j)y_ix_j^n\}$ (by \ref{eq1}).   	
  	But 	$\displaystyle\sum _{i=1}^{n}e_iy_i=1$. This gives $a\acute{r}x_j^n=\acute{r}\{r_j^mx_j^n+\displaystyle\sum _{i=1}^{n}e_ig(r_i,r_j)y_ix_j^n\}
  =\acute{r}x_j^n\{r_j^m+\displaystyle\sum _{i=1}^{n}e_ig(r_i,r_j)y_i\}=\acute{r}x_j^nr_j^m+\acute{r}x_j^n\displaystyle\sum _{i=1}^{n}e_ig(r_i,r_j)y_i$. Hence $(a-r_j^m+\displaystyle\sum _{i=1}^{n}e_ig(r_i,r_j)y_i)\acute{r}x_j^n=0$. \\
  Therefore $(a-r_j^m+\displaystyle\sum _{i=1}^{n}e_ig(r_i,r_j)y_i)(Rx_j)^n=0$. That is $GC(a-r_j^m+\displaystyle\sum _{i=1}^{n}e_ig(r_i,r_j)y_i)e_j=0$. So $a-r_j^m+\displaystyle\sum _{i=1}^{n}e_ig(r_i,r_j)y_i)\in Q$ for every $Q\in P(x_j)$ and hence $a-r_j^m\in Q$ for every $Q\in P(x_j)$. Hence $a+Q=r_j^m+Q$. That is $\hat{a}(Q)= \hat{r_j^m}(Q)=f^m(Q)$. Thus $  \hat{a}=f^m$. As $\hat{r}$ is injective, we have $f$ in injective.
   \end{proof}
   
   A ring isomorphism $\phi: R \rightarrow \acute{R} $ is said to be {\it $*$-isomorphism} if $\phi(r^*)=(\phi(r))^*$ for any $r\in R$. Now we prove that every generalized  p.q.-Baer  $*$-ring is $*$-isomorphic to the $*$- ring of sections of the sheaf.
   \begin{theorem}\label{s7tm2} Every  generalized  p.q.-Baer  $*$-ring is $*$-isomorphic to $\Gamma(\sum (R),S)$.   
\end{theorem}
  \begin{proof} Let  $\phi: R \rightarrow \Gamma(\sum (R),S)$ by $\phi(r)=\hat{r}$. Clearly $\phi$ is well defined and onto. Let $ a,b\in R$ and $Q\in \sum( R) $. Suppose that $\phi(a)=\phi(b)$, that is $\hat{a}=\hat{b}$. Thus $a-b\in Q$ for all $Q\in \sum (R)$. As $\displaystyle\bigcap_{P\in \sum(R)}P=\{0\}$, this implies that $a=b$. Hence $\phi$ is a one-one function. Further $\hat{(a+b)}(Q)=(a+b)+Q=(a+Q)+(b+Q)=\hat{a}(Q)+  \hat{b}(Q)=(\hat{a} +\hat{b})(Q)$. Therefore $\phi(a+b)=\phi(a)+\phi(b)$. Similarly, we can prove that $\phi(ab)=\phi(a)\phi(b)$. Since central strict ideals are closed under involution, $\hat{a^*}(Q)=a^*+Q=(a+Q)^*=(\hat{a}(Q))^*$. That is,  $\phi(a^*)=(\phi(a))^*$. Thus, $\phi$ is a $*$-isomorphism.
  \end{proof}
  
 \begin{remark}\label{s7rm1} If $I$ is generalized central strict ideal of  {\it generalized  p.q.-Baer  $*$-ring} then $I$ need not be equal to $\sum Re$, where the sum is taken over all central projections $e\in I$ as explained in the  following example. 
   \end{remark}  
    \begin {example} 
    Consider the ring $R=M_{2}(\mathbb Z_{2})\times M_{2}(\mathbb Z_{3})$ with multiplication and $*$ defined component wise, where $M_{2}(\mathbb Z_{2}),M_{2}(\mathbb Z_{3})$ are $*$-rings with  transpose as an involution. Let $e_{1}=(I,0) , e_{2}=(0,I)$ which are central. Clearly $e_{i}^2=e_{i} ~, e_{i}^*=e_{i}$. Thus $ e_{1} , e_{2} $ are central projections. Define $J=\{(A,0)\in R~/ ~tr(A)=0 ~in~ Z_{2}\}$. We have $tr(I)=0$ in $\mathbb Z_{2}$. Thus $e_{1}=(I,0)\in J$. For $x= (A,0)\in J ,~x^n=(A^n,0)$. The smallest central projection $e_{1}=(I,0)$ is such that  $x^ne=x^n$. So $GC(x)=e_{1}\in J$. Hence $J$ is a generalized central strict ideal. But $J \neq \sum(Re)$ because not all $(A,0)\in Re$ satisfy $tr(A)=0$.
    \end{example}
   Let $P$ be a prime ideal of ring $R$. Then $O(P)=\{a\in R~|~aRs=0~ {\rm for~ some} ~s\notin P\}=\{a\in  R~|~r(aR)\nsubseteq P\}$.  
   	 Clearly, $O(P)$ is an ideal of $R$ contained in $P$.   
   
   \begin{definition} Let $P$ be a prime ideal of ring $R$. Then $GO(P)=\{a\in R~|~(aR)^ns=0  ~ {\rm for ~ some} ~s\notin P,~ {\rm for ~some } ~n\in \mathbb{N}\}=\{a\in  R~|~r(aR)^n\nsubseteq P ~{\rm for ~ some}~ n\in \mathbb{N}\}$. 
   \end{definition}
  In the following result, we prove that $GO(P)$ is  a prime generalized central strict ideal of a generalized  p.q.-Baer  $*$-ring $R$ if $P$ is a prime ideal of $R$.
   \begin{theorem}\label{s7tm3} Let $R$ be a generalized  p.q.-Baer  $*$-ring and $P$ be a prime ideal of $R$. Then $GO(P)$ is a
   prime generalized central strict ideal of $R$.   
\end{theorem}
   \begin{proof} We have, $GO(P)=\{a\in R~|~(aR)^ns=0 ~ for ~ some ~s\notin P ,n\in \mathbb{N}\}=\{a\in  R~|~r(aR)^n\nsubseteq P~for ~ some~ n\in \mathbb{N}\}=\{a\in R~|~(1-GC(a))R \nsubseteq P\} =\{a~|~1-GC(a)\notin P\}=\{a~|~GC(a)\in P\}$.
   	Let $x\in GO(P)$ then $ GC(x)\in P$. That is $GC(GC(x))=GC(x)\in P$. Therefore $GC(x)\in GO(P)$ hence $GO(P)$ is a generalized central strict ideal. Now suppose $IJ\subseteq GO(P)$ where $I,J$ are generalized central strict ideals. Thus $IJ\subseteq P$. Therefore $I\subseteq P$ or $J\subseteq P$. Let $I\subseteq P$. Let $x\in I$. Hence $GC(x)\in I\subseteq P$. Therefore $x\in GO(P)$. That is $I\subseteq GO(P)$ and hence $GO(P)$ is the prime generalized central strict ideal of $R$.
   \end{proof}
   In the following result, we prove that any prime generalized central strict ideal of a generalized  p.q.-Baer $*$-ring $R$ is of the form $GO(P)$, for some prime ideal $P$ of $R$.
   
   \begin{theorem}\label{s7tm4} Every prime generalized central strict ideal of a generalized  p.q.-Baer $*$-ring $R$ is of the form $GO(P)$, where $P$ is a prime ideal of $R$.
   \end{theorem}
   \begin{proof} We have  $GO(P) =\{a\in R/GC(a)\in P\}$. Let $J$ be a prime generalized central strict ideal of $R$. Let $S= \{1-e~|~e ~is ~ central~ projection ~in ~J\}$. Clearly, $S$ is multiplicatively closed set with $S\bigcap~ J=\phi$. By Zorn's lemma, there exists a maximal ideal $P$ of $R$ such that $J\subseteq P , ~S\bigcap P=\phi$. Let us prove that $P$ is a prime ideal of $R$. Suppose $P_1P_2\subseteq P$ where $ P_1 , P_2$ are ideal of $R$. Let $P_1 \nsubseteq P ,P_2 \nsubseteq P $. Therefore there exists $e_1\in P_1$ such that $e_1 \notin P$ hence $1-e_1\in P$. Also  there exists $e_2\in P_2$ such that $e_2 \notin P$ hence $1-e_2\in P$. Thus $1-e_1\in P\bigcap S ,1-e_2\in P\bigcap S$. But $(1-e_1)(1-e_2)=1-(e_1+e_2-e_1e_2)\in P\bigcap S$ which is a contradiction. Therefore $P$ is a prime ideal of $R$. We will prove that $GO(P) =J$. If $x\in GO(P)$ then $GC(x)\in P$. Suppose $GC(x)\notin J$. So $1-GC(x)\in J$ which implies 
   	$1-GC(x)\in P$ which is a contradiction. Hence $GC(x)\in J \Rightarrow x^n\in J$ for some $n\in \mathbb{N}$. So $x\in J$ since $J$ is a prime generalized central strict ideal. Thus $GO(P) \subseteq J$. Now, let $y\in J$, therefore $GC(y)\in J\subseteq P$ and hence $y\in GO(P)$. That is $J\subseteq GO(P)$. Thus $GO(P) =J$.
   
\end{proof}

{\bf Disclosure statement:} The authors report there are no competing interests to declare.

 \end {document}